\PassOptionsToPackage{table}{xcolor}

\documentclass[10pt, a4paper, reqno]{amsart}

\usepackage[left=1.15in,right=1.15in,top=1.15in,bottom=1.15in]{geometry}

\usepackage{amsmath,amsthm}
\usepackage{fullpage}

\newtheorem{lemma}{Lemma}
\newtheorem{proposition}[lemma]{Proposition}
\newtheorem{theorem}[lemma]{Theorem}
\numberwithin{lemma}{section}

\theoremstyle{definition}
\newtheorem{definition}[lemma]{Definition}

\theoremstyle{definition}
\newtheorem{remark}[lemma]{Remark}

\theoremstyle{definition}

\theoremstyle{definition}
\newtheorem{example}[lemma]{Example}

\theoremstyle{theorem}

\theoremstyle{theorem}

\theoremstyle{theorem}

\author{Alexander Bastounis}
\address{University of Leicester,
	UK}
\email{ajb177@leicester.ac.uk}

\author{Felipe Cucker}
\address{City University of Hong Kong, 
	HONG KONG}
\email{macucker@gmail.com}

\author{Anders C. Hansen}
\address{University of Cambridge, UK}
\email{ach70@cam.ac.uk}

\usepackage{amssymb}

\numberwithin{equation}{section}

\usepackage{bbm}
\usepackage{enumitem}
\usepackage{times}
\usepackage{cases}

\usepackage{calc}
\usepackage{ifthen,tabularx,graphicx,multirow}
\graphicspath{ %
	{Figures/} %
	{Figures/Zoom} %
}
\usepackage{rotating}
\usepackage{hyperref}
\usepackage{tikz}

\usetikzlibrary{tikzmark}
\usepackage{ etoolbox}
\usepackage{tcolorbox}

\usepackage[algo2e]{algorithm2e}
\usepackage{cite}
\newcolumntype{Y}{>{\centering\arraybackslash}X}

\theoremstyle{definition}

\newcount\colveccount
\newcommand*\colvec[1]{
	\global\colveccount#1
	\begin{pmatrix}
		\colvecnext
	}
	\def\colvecnext#1{
		#1
		\global\advance\colveccount-1
		\ifnum\colveccount>0
		\\
		\expandafter\colvecnext
		\else
	\end{pmatrix}
	\fi
}
\makeatletter
\newcounter{framedeqn}
\AtBeginEnvironment{subequations}%
{\let\c@equation\c@framedeqn}
\makeatother

\newcommand{\real}{\mathbb{R}}
\newcommand{\nat}{\mathbb{N}}
\newcommand{\indic}{\mathbbm{1}}
\newcommand{\sg}{\mathsf{sgn}}
\newcommand{\sgn}{{\frak{s}}}
\newcommand{\argmin}{\mathop\mathrm{argmin}}

\newcommand{\supp}{\mathsf{supp}}

\newcommand{\rational}{\mathbb{Q}}

\newcommand{\gprob}{\Gamma^{\mathrm{ran}}}

\newcommand{\gprobh}{\Gamma^{\mathrm{ran},\mathrm{h}}}

\usepackage{amsfonts}
\usepackage{mathrsfs}
\usepackage{tikz}
\usepackage{float}
\usepackage{csquotes}

\newcommand{\condfsul}{\mathscr{C}_{\mathrm{UL}}}
\newcommand{\Sigmul}{\mathsf{\Sigma}_{\mathrm{UL}}}
\newcommand{\stabsupp}{\mathsf{stsp}}
\newcommand{\Sol}{\mathsf{Sol}}
\newcommand{\mUL}{\mathsf{Sol}^{\mathsf{UL}}}

\newcommand{\triple}[1]{[\![(#1)]\!]}

\allowdisplaybreaks
\newcommand{\orFindMatrix}{\mathsf{GetMatrix}}
\newcommand{\orFindVector}{\mathsf{GetVector}}
\newcommand{\orSigma}{\mathsf{Sigma}}

\newcommand{\orULasso}{\mathsf{ULasso}}
\newcommand{\orFSUL}{\mathsf{FSUL}}

\newcommand{\orPosDef}{\mathsf{PosDef}}

\newcommand{\transp}{^{\mathrm{T}}}

\newcommand{\Id}{\/\mathrm{I}}
\newcommand{\eproof}{\hfill\qed}

\def\Oh{\mathcal O}

\newcommand{\nMatSing}[2]{\|#1\|_{#2}}
\newcommand{\nMatAl}[2]{\left\|#1\right\|_{#2}}

\newcommand{\nmax}[1]{\|{#1}\|_{\max}}

\newcommand{\nTwoTwo}[1]{\nMatSing{#1}{2}}
\newcommand{\nTwoTwoAl}[1]{\nMatAl{#1}{2}}
\newcommand{\nTrMax}[1]{\triple{#1}_{\max}}
\newcommand{\nTrTwo}[1]{\triple{#1}_2}
\newcommand{\nTrOneStar}[1]{\triple{#1}_{\mathrm{S}}}

\title[A condition-based analysis of LASSO and generalised hardness of
approximation]{When can you trust feature selection? -- I:
A condition-based analysis of LASSO and generalised hardness of
approximation}

\newcommand{\mULMS}{\Xi^{\text{ms}}}
\newcommand{\Cond}{\mathsf{Cond}}
\begin{document}
	\maketitle

\begin{abstract}	
The arrival of AI techniques in computations, with the potential for hallucinations and
non-robustness, has made trustworthiness of
algorithms a focal point. However, trustworthiness of the many
classical approaches are not well understood. This is the case for feature selection, a classical problem in the
sciences, statistics, machine learning etc. Here, the LASSO
optimisation problem is standard. Despite its
widespread use, it has not been established when the output of
algorithms attempting to compute support sets of minimisers of LASSO
in order to do feature selection can be trusted. 
In this paper we establish how no (randomised) algorithm that works on all inputs can determine the correct support sets (with probability $> 1/2$)
of minimisers of LASSO when reading approximate input, regardless of
precision and computing power. However, we
define a LASSO condition number
and design an efficient algorithm for computing these support sets
provided the input data is well-posed (has finite condition number)
in time polynomial in the dimensions and logarithm of the condition number. 
For ill-posed inputs the algorithm runs forever, hence, it will never produce a wrong answer. 
Furthermore, the algorithm 
computes an upper bound for the condition number when this is
finite. Finally, for any algorithm defined on an open set containing a point with infinite condition number, 
 there is an input for which the algorithm will either run forever or produce a wrong answer. Our impossibility results stem from \emph{generalised hardness of approximation} --
within the Solvability Complexity Index (SCI) hierarchy framework --
that generalises the classical phenomenon of hardness of
approximation.
\end{abstract}

\section{Introduction}
In the wake of the many AI-based
algorithms 
throughout the society and the sciences, potentially yielding hallucinations and instabilities \cite{antun2020instabilities, jin17,
mccann2017convolutional, Choi,DezFaFr-17, hammernik2018learning, heaven2019deep, Anders2, Choi, Choi2}, the question of
trustworthiness of algorithms is now becoming a crucial topic. This is
particularly true for government regulators, where the European
Commission~\cite{EU_Commission_2020} has been particularly vocal about
its demand for trust in algorithms. However, with this new focus on
trust in algorithms comes an important question: Which of the
classical (non-AI-based) approaches are trustworthy?  For example, can
solutions to the \emph{unconstrained LASSO feature/model selection}
problem be trustworthily computed? See Remark \ref{rem:trust} for a discussion of what one means by a trustworthy algorithm. 
The now classical LASSO feature selection approach, initiated by
Tibshirani~\cite{LassoStart} is now standard in large parts of the
sciences and is defined as follows, see also \cite{SLSBook, Tibshirani_Book, Boyd, Juditsky_2012}.

\begin{definition}
For fixed $\lambda\in \rational$, $\lambda > 0$, the
\emph{unconstrained LASSO feature selection problem}
is the following: given input $\iota = (y,A)$ with $y \in \real^{m}$ and
$A \in \real^{m \times N}$, find the support of some vector in 
\begin{equation}\label{eq:UL}
   \mUL(b,U) := \argmin_{\hat{x}\in \real^{N}} \|A\hat{x}-y\|^2_2 + \lambda \|\hat{x}\|_1.
\end{equation}
The output set for input $(y,A)$ is therefore 
\begin{equation}\label{eq:LassoComp}
   \Xi(y,A) = \lbrace \supp(x) \, \vert \,
   x \in \argmin_{\hat{x} \in \real^{N}}\|A\hat{x}-y\|^2_2 
   + \lambda \|\hat{x}\|_1\rbrace
\end{equation}
and we have $\Xi(y,A)\subseteq \mathbb{B}^N$ where
$\mathbb{B}=\lbrace 0,1 \rbrace$. 
\end{definition}

\begin{remark}[Model of computation -- Inexact input]\label{rem:1}
In practice, when trying to compute an element of $\Xi(y,A)$
in~\eqref{eq:LassoComp}, we must assume that the $A$ and $y$ are given
inexactly. This is because either we have: (1) an irrational input; or
(2) the input is rational (for example $1/3$), but our computer
expresses numbers in a certain base (typically base-2); (3) the
computer uses floating-point arithmetic for which -- in many cases --
the common backward-error analysis (popularized by
Wilkinson~\cite{Wilkinson63}) translates the accumulation of round-off
in a computation into a single-perturbation of the input data. Hence,
in the sequel, we assume that algorithms access the input to
whatever finite precision desired and that all computational operations
are done exactly.
\end{remark}

The key questions we address in this paper are the following: 

\vspace{2mm}

\begin{displayquote}
\normalsize
{\it When do there exist algorithms that can compute a support set in
$\Xi(y,A)$, when $(y,A)$ are given inexactly (yet with arbitrary large
precision)? If trustworthy algorithms cannot exist for all inputs, when
can we guarantee trustworthy computations?}
\end{displayquote}

\vspace{2mm}

These questions touch on condition numbers, generalised hardness of
approximation and robust optimisation.

\begin{remark}[Trustworthiness of algorithms]\label{rem:trust}
By 'trustworthy algorithm' for a computational problem, we mean the following. If the computational problem takes only discrete values (as is the case when computing support sets of minimisers of optimisation problems), a trustworthy algorithm will always produce a correct answer -- if it halts. If the computational problem takes non-discrete values, a trustworthy algorithm produces an output -- if it halts -- that is at least $\epsilon \geq 0$ accurate in some predefined metric -- for any $\epsilon > 0$ given as input to the algorithm. 
\end{remark}

\subsection{Condition and trustworthiness} 
The phenomenon of small imprecisions on the input data leading to
substantial errors on the output of an algorithm is a classical
challenge in numerical analysis.  Indeed, this magnification of the
input error can in many cases be gauged by a notion
of \emph{condition}, which needs to be defined for each individual
problem. This is crucial in the developments of trustworthy algorithms
--- that is algorithms with guaranteed error bounds.  A {\em condition
number} $\Cond$ is a map from the space of inputs to the interval
$[0,\infty]$ that defines the sensitivity to small perturbations, with
high values indicating sensitive inputs.  Condition numbers were
introduced independently by Turing~\cite{Turing48} and von Neumann and
Goldstine~\cite{vNGo47} in a pair of papers widely considered the
birth certificate of contemporary numerical analysis. The goal was to
understand the effects of finite precision in linear system solving to
ensure trustworthy algorithms. Shortly after, condition numbers took
also a role in the computation of complexity
bounds. See~\cite[Overture]{Condition} for a global picture about
this.

\subsection{Generalised hardness of approximation (GHA) and robust
optimisation}

The program on robust optimisation, pioneered by A. Ben-Tal, L. El
Ghaoui \& Nemirovski~\cite{Nemirovski_robust, NemirovskiLRob,
Nemirovski_robust2} among others, provides a powerful mathematical
framework that addresses the challenge of optimisation given
inaccuracies and uncertainty in the input. This model, as argued in
Remark~\ref{rem:1}, is more realistic in view of real life
computations --- compared to models assuming exact computations ---
and aligns with S. Smale's call for ``[Computational] models which
process approximate inputs and which permit round-off computations''
in the list of problems for the 21st century~\cite{MathFrontiersPerspectives}.
Recent developments in
this area intersect with generalisations of the phenomenon of hardness
of approximation~\cite{Arora2007}, namely generalised hardness of
approximation (GHA), initiated in~\cite{opt_big} (see
also~\cite{CSBook, gazdag2022generalised, comp_stable_NN22, paradox22}
and Problem 5 (J. Lagarias) in~\cite{AIM} for further results on
GHA). GHA in optimisation is the phenomenon where one can easily
compute an $\epsilon$-approximation to a minimiser of the optimisation
problem when $\epsilon > \epsilon_0 > 0$, but for $\epsilon
< \epsilon_0$ (the approximation threshold) it suddenly becomes
impossible regardless of computing power and accuracy on the
input. This phase transition phenomenon was recently established
\cite{opt_big, AIM} for
computing minimisers of the LASSO problem~\eqref{eq:UL}, and our
impossibility results build on this framework and extend these
results. Robust optimisation is classically concerned with the problem
of approximating the optimal value of the objective function. However, a theory of robust
optimisation for computing minimisers will necessarily include the GHA
phenomenon.

\begin{remark}[{\bf Condition and GHA}]
 Condition numbers are well-known to often provide sufficient
conditions for the existence of trustworthy algorithms in
optimisation~\cite{Condition, felipe_cond_01, Renegar1, Renegar2, renegar1988polynomial,
renegar2001mathematical, Renegar96}.  J. Renegar's seminal work \cite{Renegar1, Renegar2, renegar1988polynomial,
renegar2001mathematical, Renegar96} on condition numbers in optimisation start with a definition of ill-posed problems, and then a natural condition number is 1/(distance to the closest ill-posed problem).  
New developments demonstrate how, in certain
situations, the necessity of having finite standard condition numbers
in order to obtain trustworthy algorithms is not needed. Indeed, the
theory of GHA~\cite{opt_big, CSBook, gazdag2022generalised,
comp_stable_NN22, paradox22} demonstrate examples of classes of
optimisation problems in the sciences that have infinite condition
numbers that are standard in optimisation, yet efficient and
trustworthy algorithms can handle the problems.

For linear programs and basis pursuit problems restricted to an input class $\Omega$ with inputs $\iota = (y,A)$, where $A \in \mathbb{R}^{m\times N}$ satisfies the robust nullspace condition \cite{CSBook} of order $s \geq 1$, with fixed parameters $\tau > 0$ and $\rho \in (0,1)$, and $y = Ax$, where $x$ is $s$-sparse, we have the following phenomenon \cite{opt_big}. There exists an algorithm that can compute approximate minimisers of the optimisation problem when $A$ and $y$ are given with inexact input. Moreover, such approximate minimisers can be computed in polynomial time in the number of variables and $\log(\epsilon^{-1})$ where $\epsilon > 0$ is the bound on the error of the approximate minimiser produced by the algorithm. As is standard in linear programming \cite{Condition}, the set of ill-conditioned problems is the collection of problems with several minimisers, yielding a condition number $\mathrm{Cond}$ for minimisers of linear programs. However, the set $\Omega$ contains input elements $\iota = (y,A)$ for which the condition number  $\mathrm{Cond}(\iota) = \infty$. In particular, there exist polynomial time algorithms for vast input classes in the sciences for which classical condition numbers are infinite. 

This suggests that the
marriage of condition theory and GHA provide a much more refined
theoretical framework designed to understand when trustworthy and
efficient algorithms can be designed.  This paper is the first step in
this direction.
\end{remark}

\section{Main Result}

Our paper continues the developments of condition in connection with
GHA and provides both upper and lower bounds. We define in
Section~\ref{section:ULResults} a condition number $\condfsul$ that
gauges the impact of numerical errors when using feature selection via
unconstrained LASSO. This quantity follows classical ideas of
condition as the inverse of the distance to ill-posedness (see for
instance~\cite{ChC02} for a close ancestor of $\condfsul$ for linear
programming).  We provide an alternative definition of $\condfsul$
based on the Karush-Kuhn-Tucker (KKT) conditions in
Section~\ref{section:alternativeCharacterisation}. Alongside the
condition of the data at hand, the cost of solving LASSO also depends
of the size of the data.  We will measure this size with the following
``truncated norms:''
\begin{equation*}
 \triple{b,U}_{\max} := \max \left\{\|U\|_{\max},\, \|b\|_{\infty},\, 1\right\}, \nTrOneStar{y,A}:= \max\left\{\sum_{i=1}^{m}\sum_{j=1}^{N} |A_{ij}|,\,
 \sum_{i=1}^{m}  |y_i|,\,1\right\}
\end{equation*}
where $\|U\|_{\max}=\max_{i,j}|U_{ij}|$ and $\|b\|_\infty =\max_{i}|b_i|$.

Our main result is the following (we will make the
meaning of ``variable-precision approximations'' precise in 
Section~\ref{sec:model}).

\begin{theorem}\label{thm:fsulcomp}
Consider the condition number $\condfsul(b,U)$ defined
in~\eqref{eq:CondDefinition}.
\begin{itemize}
\item[(1)]  We exhibit an algorithm $\Gamma$ which, for any input pair 
$(b,U)\in\real^m\times\real^{m\times N}$, reads variable-precision 
approximations of $(b,U)$. If $\condfsul(b,U) < \infty$
%(see~\eqref{eq:CondDefinition})
then the algorithm halts and returns a correct value in $\Xi(b,U)$.
The cost of
this computation is
\begin{equation*}
   \Oh\left\{N^{3}\left[\log_2\left(N^2 \triple{b,U}_{\max}^2
   \condfsul(b,U)\right)\right]^2\right\}
\end{equation*}
\label{item:fsulcomppos}
and the maximum number of digits the algorithm accesses is bounded by 
\[
\Oh(\left\lceil \log_2\left(\max\{\lambda + \lambda^{-1},N,\nTrOneStar{b,U},
\condfsul(b,U)\}\right)\right\rceil).
\]	
If, instead, $\condfsul(b,U)=\infty$ then the algorithm runs forever.

\item[(2)]
The condition number $\condfsul(b,U)$ can be estimated in the
following sense: There exists an algorithm that provides an upper bound
on $\condfsul(b,U)$, when it is finite, and runs forever when
$\condfsul(b,U) = \infty$.

\item[(3)]
If $\Omega\subseteq \real^m\times\real^{m\times N}$ is an open set and
there is a $(b,U)\in\Omega$ with $\condfsul(b,U) = \infty$ then there
is no algorithm that, for all  inputs $(y,A)\in\Omega$, computes an element
of $\Xi(y,A)$ given approximations to $(y,A)\in \Omega$.  Moreover, for any
randomised algorithm $\Gamma^{\mathrm{ran}}$ that always halts and any $p>1/2$, there exists
a $(y,A)\in\Omega$ and an approximate representation $(\tilde{y},\tilde{A})$ (see \S \ref{sec:SCI})
 of $(y,A)$ so that $\Gamma^{\mathrm{ran}}(\tilde{y},\tilde{A}) \notin \Xi(y,A)$ with probability
at least $p$.

 If $(b,U)\in\Omega$ is computable, then the failure point $(y,A)\in\Omega$ above can be made computable.  
\end{itemize}
\label{item:fsulcompneg}
\end{theorem}

\begin{remark}[{\bf Our algorithm never produces wrong outputs}] 
In terms of trustworthiness -- which is the main topic of this paper -- our algorithm will never make a mistake. In the cases where it fails, it will simply run forever. Note that according to Theorem \ref{thm:fsulcomp}, this is optimal for any algorithm that will work on open sets of inputs, as the alternative to not halting is producing a wrong output when $\condfsul(b,U) = \infty$. 
\end{remark}

The complexity bound presented in Theorem~\ref{thm:fsulcomp} is
pleasantly low, exhibiting a cubic dependence on $N$ and
logarithmic dependence on both the size of the data and its condition
number. It is worth highlighting that (3) within
Theorem~\ref{thm:fsulcomp} implies that our defined condition number
effectively captures the challenging aspects of solving unconstraint
LASSO with open input sets. Indeed, open sets of inputs containing
elements with an infinite condition number inherently preclude the
application of reliable algorithms. Our current approach does not
provide a means to determine when the condition number becomes
infinite. Generally, it is widely believed that discerning
with finite precision whether a data has finite or infinite
condition number is, in most cases, not possible
(cf.~\cite{Renegar94}). The estimate of this condition number
is, for a broad class of problems, as challenging as solving the
underlying problem itself~\cite{ChC05}.
This insight was promptly recognized by von
Neumann and Goldstine, prompting them to address this issue in their
sequel~\cite{vNGo51} to~\cite{vNGo47}.

The proposed solution in that sequel, which has since become a commonly
followed approach, involves imbuing the data space with a probability
measure and deriving probabilistic estimates for the condition
number. These estimates, in turn, lead to probabilistic bounds on
complexity and accuracy. A sequel~\cite{BCH2} of our paper
proceeds along these lines.

\begin{remark}[{\bf Condition number and robust optimisation}] 
Both the fields of condition in optimisation and robust optimisation
share a common objective: ensuring stable and precise computations
even in the presence of imprecise input data. However, historically,
these two domains have often remained somewhat distinct within
mathematics.  A reason for this is that robust optimisation has
focused on computing the optimal value and not the minimisers.
As our negative results in Theorem~\ref{thm:fsulcomp} demonstrate,
providing a theory about properties of minimisers within robust
optimisation, necessarily involves condition numbers.  In particular,
our results illustrate how a theory of robust feature selection through
optimization inherently links robust optimization and the theory of
condition.
\end{remark}

\subsection{Connection to previous work} Below follows an account of
the connection to different areas and works that are crucial for the paper. 

\begin{itemize}[leftmargin=5pt]
\item[] \emph{Condition in optimisation:} Condition numbers in
computational mathematics and numerical analysis have been a
mainstay~\cite{Higham96, Cucker_Smale97} in order to secure
trustworthy algorithms that are accurate and stable. In optimisation
J. Renegar has pioneered the theory of condition
numbers~\cite{Renegar1, Renegar2, renegar1988polynomial,
renegar2001mathematical, Renegar96} both from the perspective of
stability and accuracy, but also from the perspective of efficiency of
algorithms. All of these perspectives are covered in detail
in~\cite{Condition}. In addition, we want to mention the work of
J. Pe{\~n}a~\cite{Pena2002, Pena2001} as well as D. Amelunxen,
M. Lotz, J. Walvin~\cite{Lotz2020}, and D. Amelunxen, M. Lotz,
M. McCoy, J. Tropp~\cite{Lotz2014} see also~\cite{ChCP07,
felipe_cond_01, ChC02}.
\vspace{1mm}

\item[] \emph{GHA and robust optimisation:}
GHA~\cite{opt_big, CSBook, gazdag2022generalised, comp_stable_NN22,
paradox22} is in spirit (although mathematically very different) close
to hardness of approximation in computer
science~\cite{Arora2007}. However, GHA in optimisation can be viewed
as a part of the program on robust optimisation (A. Ben-Tal, L. El
Ghaoui \& Nemirovski~\cite{Nemirovski_robust, NemirovskiLRob,
Nemirovski_robust2}) for computing minimisers. It is also a
part of the
greater program on the mathematics behind the Solvability Complexity
Index (SCI) hierarchy, see for example the work by J. Ben-Artzi,
M. Colbrook, M. Marletta~\cite{Hansen_JAMS, Ben_Artzi2022, SCI,
Matt1}.

\vspace{1mm}

\item[] \emph{Trustworthy algorithms and computer assisted proofs:}
Trustworthy algorithms in optimisation go beyond scientific computing
and have important implications in computer assisted proofs in
mathematics, where T. Hales' proof of Kepler's
conjecture~\cite{Hales1, Hales2} is a star example.  The intricate
computer assisted proof relies on computing around 50,000 linear
programs with irrational inputs, which leads to the crucial problem of
computing with inexact inputs. The reader may also want to
consult~\cite{AIM}, in particular Problem 2 (T. Hou) and Problem 5 (J. Lagarias) on  

a paper that discusses the tradition of
developing algorithms that are 100\% trustworthy and even suitable for
computer assisted proofs.

\vspace{1mm}

\item[] \emph{Algorithms for computing minimisers of LASSO:}  There is, of course, a
variety of algorithms suitable for the LASSO
problem. We refer to the review articles by Nesterov \& Nemirovski \cite{Nesterov_Nemirovski_Acta} and Chambolle \& Pock \cite{chambolle_pock_2016}, and the references therein for a more complete overview. See also the work by Beck \& Teboulle \cite{Fista} and Wright, Nowak \&  Figueiredo \cite{Mario_Lasso}, and the references therein, as well as \cite{Boyd, wright2022optimization, CSBook}. However, while these algorithms can compute approximations to the objective function, they cannot -- in general -- compute the support sets of minimisers of LASSO. 
\end{itemize}

\section{Numerical Examples -- Failure of modern algorithms}

\begin{table}
\begin{center}
\begin{tabular}{|c|c|c|c|c|}
\hline 
$\epsilon$ & $x_1$ & $x_2$ & Features computed &
Features correctly computed\\
\hline  \hline
$10^{-1}$ & $0$ & $0.95$ & $\{2\}$ & Yes\\\hline 
$10^{-2}$ & $0$ & $0.95$ & $\{2\}$ & Yes\\\hline 
$10^{-3}$ & $0$ & $0.95$ & $\{2\}$ & Yes\\\hline 
$10^{-4}$ & $0.45$ & $0.5001$ & $\{1,2\}$ & No\\\hline  
$10^{-5}$ & $0.9$ & $0.0500$ & $\{1,2\}$ & No\\\hline  
$10^{-6}$ & $0.945$ & $0$ & $\{1\}$ & No\\\hline
\end{tabular} 
\end{center}
\caption{The computation done in Example~\ref{Example:2DDeterministic}. Note
that there is an error for $\epsilon = 10^{-4},10^{-5}$ or
$10^{-6}$.
\label{table:SimpleDeterministicExample}}
\end{table}

Our first numerical example shows how this idea can work in practice.

\begin{example}\label{Example:2DDeterministic}
Assume that a dependent variable $b$ depends on two features $b_1,b_2$
in the following way: $b = u_1/(1-\epsilon) = u_2$ where the
parameter $\epsilon \in (0,1)$. In particular, $b$ is strongly
correlated with $u_2$ and this is the best predictor (using the lasso) for
$b$. Suppose that we observe $b,u_1$ and $u_2$ and obtain the
measurements $u = 1$, and
$U = \begin{pmatrix} u_1 &
  u_2 \end{pmatrix}=\begin{pmatrix} 
  1- \epsilon & 1 \end{pmatrix}$.
The LASSO problem with input
$(b,U)$ has a unique solution $x^*$ with $\supp(x^*) = \{2\}$
provided $\lambda < 2$.
We tested the accuracy of LASSO solvers under finite precision
by computing solutions
to the LASSO problem with $\lambda = 10^{-1}$ in the following
way: first, we used MATLAB's lasso routine to attempt to
find a minimiser. Next, we set any values of
this minimiser that were smaller
in magnitude than $10^{-2}$ to 0 (not doing so would systematically
lead to minimisers with full support). 
We then computed the support of the resulting vector $x$.
Table~\ref{table:SimpleDeterministicExample} presents the
results for the computed vector $x$ and its support.
For small values of $\epsilon$, the execution was
unable to identify the zero component of the true solution $x^*$ 
of the unconstrained LASSO problem. There is, of course, a
variety of algorithms suitable for the LASSO
problem~\cite{Chambolle_2011, chambolle_pock_2016}, however,
Theorem~\ref{thm:fsulcomp} is universal, and thus for any algorithm there will be inputs for which the algorithm will fail.
\end{example}

An obvious extension of Example~\ref{Example:2DDeterministic} is to
move from the deterministic inputs of
Example~\ref{Example:2DDeterministic} to inputs chosen
according to a random distribution. We thus look at a simple
random example, where we can once again compute by hand the
true solution and thus contrast with the output of the
algorithm. 

\begin{example}\label{Example:NDRandom}
Again we consider the single measurement case and set $b=1$, and $\lambda = 10^{-2}$ but this time we assume that we observe $N$
features, $u_1,\ldots,u_N$, which are randomly and independently
drawn according to either the exponential distribution with parameter $1$, the normal distribution with mean $1$ variance $10^{-4}$, and the uniform distribution on $(0,1)$. This purposefully simplified situation is considered because it is easy to derive the true solution to the
feature set induced by the solution of~\eqref{eq:UL}. We compare this answer (the ground truth) with the following procedure: we use Matlab's lasso routine to attempt to compute an element $x$ in $\mUL(b,U)$ and then set any values of $x$ larger in absolute value than a parameter `threshold' to 0 and consider the resulting vector's support.

 We do this over five hundred
randomly generated instances $U=(u_1,\ldots,u_N)$
for each choice of $N \in \{10,20,\dotsc,10010\}$
inclusive. The results are presented in
Figure~\ref{fig:ExampleNDRandomResult}. We see that
for large $N$, the algorithm is more accurate when data are
drawn from an exponential distribution instead of this normal
distribution and similarly the algorithm is more accurate when
data are drawn from a normal distribution instead of a uniform
one.
\end{example}

\begin{figure}
	\begin{center}
		\begin{tabular}{c c}
			\includegraphics[scale = 0.8]{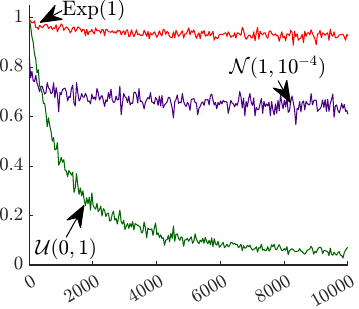} & \includegraphics[scale = 0.8]{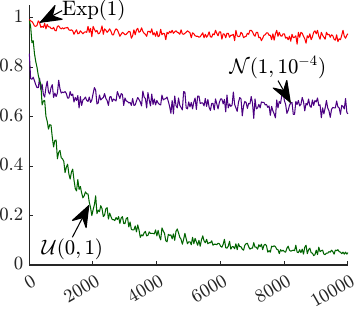} 
		\end{tabular} 
	\end{center}
	\caption{
		Outputs of the computation done in Example~\ref{Example:NDRandom}  -- where each entry of $U$ is IID  according to the distributions $\mathcal{U}(a,b)$ (uniform), $\mathrm{Exp}(\nu)$ (exponential) and $\mathcal{N}(\mu, \sigma^2)$ (normal). The task is to compute the support set of a LASSO minimiser i.e. an element in $\Xi(b,U)$ from \eqref{eq:LassoComp} with $\lambda = 10^{-2}$. The horizontal axis represents the dimension $N$. The vertical axis represents the success rate $\frac{\# \text{ of successes }}{\# \text{ of trials }}$ with threshold value -- see the text accompanying in Example \ref{Example:NDRandom} -- set to $10^{-3}$ and $10^{-12}$ for the left and right figures respectively.}
	\label{fig:ExampleNDRandomResult}
\end{figure}

\section{Preliminaries on the computational model}
\label{sec:model}

To talk about algorithms and complexity it is necessary to fix a
computational model. In our case, we can take either a BSS
machine~\cite{BSS89,BCSS98} or a Turing machine. The only assumption
we will do on these machines is that they can access arbitrarily
precise approximations to the input data (and that the machine
``knows" the precision of these approximations). It is clear that
otherwise, there would be no hope to return a correct output. It is
also clear that the cost of accessing one such approximation will have
to depend of the precision required. We next specify this.

We will assume that, for any pair $(b,U)\in\real^m\times\real^{m\times
N}$ we want to solve, the machine is given as input two procedures
$\orFindMatrix^U(~)$, $\orFindVector^b(~)$ taking themselves as input
a natural number $n \in \mathbb{N}$ which return some matrix $U^n$ and
vector $b^n$ respectively such that
\begin{equation}\label{eq:OracleGuarantee}
|U^{n}_{ij} - U_{ij}| \leq 2^{-n},\quad |b^n_i - b_i| \leq 2^{-n} \quad
\text{for } i=1,2,\dotsc,m \text{ and } j=1,2,\dotsc,N.
\end{equation}
The procedures  $\orFindMatrix^U(~)$, $\orFindVector^b(~)$ are
{\em black boxes} or {\em oracles}. We do not suppose any specific
implementation for them, merely their correctness.

Of course, for any fixed $(b,U)$ there are multiple choices of $b^n$
and $U^n$ satisfying~\eqref{eq:OracleGuarantee}. Crucially, the
algorithm must work for any such choice. That is to say, the
algorithm has to work with whichever choice of $b^n$ or $U^n$ the
oracles return, it can only rely on the fact that
$(b^n,U^n)$ approximates $(b,U)$ with a precision given
by~\eqref{eq:OracleGuarantee}.
But to be consistent with both the BSS and Turing
choices, we want $U^n$ and $b^n$ to have rational
entries\footnote{This consistency refers to our algorithm in
Theorem~\ref{thm:fsulcomp}(1) which can be thought either as a Turing
or a BSS algorithm as it works over rational numbers. We note, however,
that the complexity bound in the statement refers to the BSS setting. Also,
part~(3) of the theorem holds for both Turing and BSS algorithms.}.
And for complexity reasons (we will rely on existing algorithms for
convex quadratic programming that take rational inputs) we
also want to control the bit-size of the entries of $b^n$ and $U^n$.
We now observe that for any $x\in\real$ and $n\geq 1$, we can compute
a rational number $\tilde{x}$ such that $|\tilde{x}-x|\leq 2^{-n}$
with
\begin{equation}\label{eq:bitsize}
   \textrm{bit-size}(\tilde{x})\le \lceil\log_2(1+|x|)\rceil+n.
\end{equation}
Hence, we will strengthen assumption~\eqref{eq:OracleGuarantee} with
the bounds~\eqref{eq:bitsize} for the entries of $U^n_{ij}$ and $b^n_i$.

Furthermore, the computation of $\tilde{x}$ above can be done with a cost
of $\Oh(\lceil\log_2(1+|x|)\rceil+n)$ arithmetic operations. This
suggests a natural cost of $\Oh(mN(\log_2(\|U\|_{\max}+1)+n))$ 
for a call to $\orFindMatrix^U(n)$ and one of 
$\Oh(m(\log_2(\|b\|_{\infty}+1)+n))$ for a call to $\orFindVector^b(n)$.
We will therefore adopt these costs. 

We also make the assumption that $\lambda\in\rational$ and that the algorithm
can access the value of $\lambda$ exactly. This is because this parameter is
usually configurable by an end user working on the feature selection problem,
and can thus be chosen to be a rational number. This is not the case for $b$
and $U$ which are usually taken from real world datasets and may be
irrational. It goes without saying, at least in the Turing model, the bit-size
$p(\lambda)$ of $\lambda$ affects the cost of the computation. But for a fixed
$\lambda$ the bound in Theorem~\ref{thm:fsulcomp} holds. Nonetheless,
we will explicitly quantify this effect in Section~\ref{sec:proofs}.

In addition to $\triple{y,A}_{\max}$ we will also use the $\ell^p$-norm
$\|y\|_{p}$ of $y$, the operator norms 
$\|A \|_{qr} = \sup_{\|x\|_q = 1} \|Ax\|_r$  (writing $\|A\|_q$ when $q=r$),
and the additional `truncated norm' $\nTrTwo{y,A}:=\max \left\{\|A\|_2,\|y\|_2,1\right\}$.

\section{Standard facts about the unconstrained LASSO}
\label{section:BasicFactsAboutUL}

In this section we will describe some well-known facts about the
unconstrained LASSO selector.  We make no claims to novelty ---
instead, this section exists only to supplement the rest of the
article.

For a pair $(y,A)$, let $\mUL(y,A)$ be the set of solutions in $\real^N$ 
to unconstrained LASSO with input $(y,A)$. More
precisely, we set
\[
  \mUL(y,A) := \argmin\limits_{\hat{x}\in \real^{N}}
  \|A\hat{x}-y\|^2_2 + \lambda \|\hat{x}\|_1.
\]
Note that $\mUL(y,A)$ is a set and not necessarily single valued:
a priori, multiple
vectors may minimise the unconstrained LASSO functional. Then, we recall, 
$\Xi(y,A)=\{\supp(x)\mid x\in\mUL(y,A)\}$.

The following facts are well-known:
\begin{description}
\item[(UL1)] If $x \in \mUL(y,A)$ we must have
$\|Ax-y\|^2_2,\lambda \|x\|_1\leq \|Ax-y\|^2_2 +\lambda \|x\|_1\leq \|y\|^2_2$.
Thus $\mUL(y,A)$ can be written equivalently as 
\[
  \mUL(y,A) = \argmin_{{\hat{x}\in \real^{N},\,
  \|\hat{x}\|_1 \leq \lambda^{-1}\|y\|^2_2}} \|A\hat{x}-y\|^2_2
  + \lambda \|\hat{x}\|_1.
\]
For each fixed $(y,A)$, this is a minimisation problem of a continuous
function (of $\hat{x}$) over a non-empty compact region. Thus
$\mUL(y,A)$ is always non-empty and compact.  This situation
is different from that of linear programming, where there
exist inputs without optimal solutions.  \item[(UL2)] As the
objective function (that is to say, the function mapping
$\hat{x}$ to $\|A\hat{x}-y\|^2_2 + \lambda \|\hat{x}\|_1$) is
convex, the unconstrained LASSO problem is convex.
	
\item[(UL3)]
As mentioned before, it is not necessary that the unconstrained LASSO
problem has a unique minimiser. However, as the problem is
convex and the feasible region is non-empty, either there is a
unique minimiser or there are infinitely many
minimisers. Understanding uniqueness will prove to be
important to this paper and has been examined in detail
in~\cite{LASSOUNIQUE}.
	
\item[(UL4)] As argued
in~\cite{LASSOUNIQUE}, if $x^1,x^2 \in \mUL(y,A)$ then $Ax^1 =Ax^2$.
Thus $\|Ax^1-y\|^2_2 = \|Ax^2-y\|^2_2$ and so
$\|x^1\|_1
= \|x^2\|_1$. \label{fact:ULConservesTargetAndNorm}
	
\item[(UL5)\label{fact:ULKKTNecessarySufficient}]
The KKT conditions are both necessary and sufficient. That is, 
\[
x \in \mUL(y,A) \iff 
2A^*(Ax-y) =
-\lambda\, \sgn(x)
\]
where
\begin{equation*}
(\sgn(x))_i
\begin{cases}
  = 1 & \text{ if } x_i > 0\\
  = -1 & \text{ if } x_i < 0\\
  \in [-1,1] & \text{ if } x_i = 0.
\end{cases}
\end{equation*} 
In particular, combining this with (UL4) we see that $\sgn(x)$ is
constant over all $x \in \mUL(y,A)$.
\end{description} 

\begin{remark}
Note that $\sgn(x)$ is closely related to the sign function $\sg$,
but is multivalued when applied to $0$ whereas $\sg(0)$ is
traditionally defined to be $0$. In fact, $\sgn(x)$ is the
convex subdifferential of the $\ell^1$ norm applied at~$x$.
\end{remark}

\section{Conditioning of unconstrained LASSO}\label{section:ULResults}

It is sensible to define a quantity that measures how variations in
the input to a LASSO problem affect the support of a solution. We
thus define the \emph{stability support}.

\begin{definition}
The \emph{stability support} of a pair $(y,A)$ is defined as 
\begin{align*}
\stabsupp(y,A):=\inf \Big\lbrace& \delta \geq 0 \, \big\vert \, 
  \exists\,  \tilde{y} \in \real^{m},\tilde{A} \in \real^{m \times N},
  x\in \mUL(y,A), \mbox{ and } \tilde{x} \in \mUL(\tilde{y},\tilde{A})\\ 
  &\!\!\mbox{ such that } \|\tilde{y} - y\|_{\infty}, \nmax{A - \tilde{A}}
  \leq \delta \mbox{ and }  \supp(x) \neq \supp(\tilde{x})\Big\rbrace.
\end{align*}
\end{definition} 
The stability support is therefore the {\em distance to support change}.  
If $\stabsupp(y,A)>0$ then there exists $S\in{\mathbb B}^N$ such that 
$\Xi(y,A)=\{S\}$. Furthermore, for all pairs $(y',A')$ in a ball
(w.r.t.~the max distance) of radius $\stabsupp(y,A)$ around $(y,A)$
we have $\Xi(y',A')=\{S\}$. If, instead, $\stabsupp(y,A)=0$ then there are
arbitrarily small perturbations of $(y,A)$ which yield LASSO solutions with
different support. 

The notion of stability support leads to a natural definition 
of condition for the unconstrained LASSO feature selection problem.

\begin{definition}
For an input $(y,A)$ to UL feature selection we define the {\em condition
number} $\condfsul(y,A)$ to be
\begin{equation}\label{eq:CondDefinition}
  \condfsul(y,A) =
  \begin{cases} 
   (\stabsupp(y,A))^{-1} & \text{if } \stabsupp(y,A) \neq 0\\
   \infty & \text{otherwise.}
  \end{cases}
\end{equation}
The set $\Sigmul:=\{(y,A)\mid \stabsupp(y,A)=0\}$ is the
{\em set of ill-posed inputs}. 
\end{definition}

\begin{remark}\label{rem:CondScaleInvariance}
For scale-invariant problems it is common to define condition as the
normalized inverse of the distance to ill-posedness. That is,
for a data $a$, as $\|a\|/d(a,\Sigma)$ where $\Sigma$ denotes
the set of ill-posed inputs, $\|~\|$ is some norm measuring
the size of the input data and $d$ is some distance metric,
usually the one induced by $\|~\|$  
(see~\cite[\S6.1]{Condition} for a discussion on this). The
fact that UL is not scale invariant explains why the condition
number $\condfsul$ is not normalized. To better understand
this issue it is worth considering a simple example. Let us
consider the input $(y,A) = (0,0)$. For any $\lambda > 0$, it
is easy to see that $\mUL(y,A) = \{0\}$. Moreover, if both
$\|y'-y\|_{\infty} =\|y'\|_{\infty} \leq \epsilon$ and
$\nmax{A'-A} =\nmax{A'}\leq \epsilon$ then
$\|A'0-A'y'\|_{\infty} =\|A'y'\|_{\infty} < \lambda$, provided
that $\epsilon \leq \sqrt{\lambda N^{-1}}$. Thus $0$ satisfies
the KKT conditions and so $0 \in \mUL(y',A')$. Since all LASSO
solutions have the same $\ell^1$ norm (by~(UL4)), we conclude
that $\mUL(y',A')=\{0\}$. Thus $\stabsupp(y,A) > 0$ and so
the input $(0,0)$ should be considered well-posed.  If we were
to define the condition number as a norm over a distance, we
would have that $\condfsul(y,A) = 0$. Such perfect
conditioning would be unhelpful in understanding how close the
data is to instability. Hence, we do not define the
condition number in the traditional way. We instead opt to
define $\condfsul$ as in~\eqref{eq:CondDefinition} so that the
condition of $(0,0)$ is both positive and finite.
\end{remark}

From the definition, it is immediately obvious that if
$(y,A) \notin \Sigmul$ then $\Xi(y,A)$ must be single valued. We
conclude this section by noting that this property transfers to
multivaluedness of $\mUL$. We will prove this result in
Section~\ref{sec:proofs}. 

\begin{proposition}\label{prop:goodSetOneSolution}
If $(y,A) \notin \Sigmul$ then $|\mUL(y,A)|=1$.
\end{proposition}

\section{An alternative characterisation of the condition number}
\label{section:alternativeCharacterisation}

In this section we present a different characterisations of
ill-posedness based on three values, $\sigma_1,\sigma_2,\sigma_3$, which
we can use to computationally approximate the condition number.

\begin{definition}\label{def:sigma1sigma2sigma3}
For a pair $(y,A)$, we write (with the convention that if $M$ is
non-invertible, $\nTwoTwo{M^{-1}}:=\infty$ and so $\nTwoTwo{M^{-1}}^{-1} = 0$),
\begin{align*}
\sigma_1(y,A)&:= \inf \lbrace t \, \vert \, \exists x \in \mUL(y,A)
\text{ with } 
\|A_{S^{\mathsf{c}}}^*(Ax-y)\|_{\infty}= \lambda/2 - t, S = \supp(x) \rbrace, \\
\sigma_2(y,A)&:=   \inf \lbrace \nTwoTwo{(A^*_SA_S)^{-1}}^{-1} \, \vert \,
\exists x \in \mUL(y,A) \text{ with } S = \supp(x) \rbrace, \\
\sigma_3(y,A)&:= \inf \lbrace t \, \vert \, \exists
i \in \lbrace 1,2,\dotsc,N \rbrace  
\text{ and } x \in \mUL(y,A) \text{ such that }  0 < |x_i| \leq t \rbrace.
\end{align*}
where, for the empty-set $\varnothing$, we interpret
$\|A^*_{\varnothing}(Ax-y)\|_{\infty} = 0$, we treat $A^*_\varnothing A_\varnothing$
as invertible with $\nTwoTwo{(A^*_\varnothing A_\varnothing)^{-1}}^{-1} = \infty$,
and we set $\inf \varnothing = \infty$. 
\end{definition}

To get an intuition of what these $\sigma$s gauge assume momentarily
that Proposition~\ref{prop:goodSetOneSolution} holds and let
$\mUL(y,A)=\{x\}$. Then $\sigma_1$ gauges how close the non-support
set of $x$ is to violating the KKT conditions; it is small if
$\|A^*_{S^{\mathsf{c}}}(Ax-y)\|_{\infty}$ is close to
$\lambda/2$. Similarly, $\sigma_2$ is small if $\nTwoTwo{(A^*_S
A_S)^{-1}}$ is large for the support set $S$ of $x$ (and thus
$A^*_SA_S$ is close to being non-invertible), and $\sigma_3$ is small
if $x$ has a small component on its support.

We combine each of $\sigma_1, \sigma_2$ and $\sigma_3$ into a single
quantity as follows,
\[
\sigma(y,A) := \min\left\{\sigma_1(y,A),\sigma_2(y,A)^2,\sigma_3(y,A)\right\}
\] 

At first glance, it might seem that computing $\sigma$ is difficult
owing to the infima in Definition~\ref{def:sigma1sigma2sigma3}. The
next proposition shows how these infima can be removed so that
$\sigma$ can be easily calculated given an $x \in \mUL(y,A)$.

\begin{proposition}\label{prop:sigma1MinimiserCompute}
Let $x \in \mUL(y,A)$ have support $S$. 
\begin{enumerate} 
\item \label{prop:sigma1MinimiserSgNZ}
If $\|A_{S^{\mathsf{c}}}^*(Ax-y)\|_{\infty} < \lambda/2$ and $A^*_SA_S$ is invertible, 
then $|\mUL(y,A)|=1$ and so
\begin{gather*}
\sigma_1(y,A)= \lambda/2 - \|A_{S^{\mathsf{c}}}^*(Ax-y)\|_{\infty}, \quad
\sigma_2(y,A)= \nTwoTwo{(A^*_SA_S)^{-1}}^{-1} \\
\sigma_3(y,A)= \inf \lbrace |x_i| \,\vert \, i \in S \rbrace.
\end{gather*}
\item 
If instead $\|A_{S^{\mathsf{c}}}^*(Ax-y)\|_{\infty} = \lambda/2$ or $A^*_SA_S$
is not invertible, then $\sigma(y,A) = 0$. \label{prop:sigma1MinimiserSgZ}
\end{enumerate}
\end{proposition}

We are now interested in the relation between $\sigma$ and $\stabsupp$.
The following proposition provides a bound from above from $\stabsupp$
in terms of $\sigma$.

\begin{proposition}\label{proposition:rhofssigmaub}
For $(y,A)\in\real^m\times\real^{m\times N}$, $\stabsupp(y,A)$ is bounded
above by $\sigma(y,A)$ as follows,
\begin{enumerate}[label=(\arabic*)]
\item 
For $\sigma_1(y,A) <  \lambda/4$, we have 
$\stabsupp(y,A) \leq \frac{4\nmax{A} \sigma_1(y,A)}{\lambda}$.  
\label{lemma:sigma1fscond}
\item 
We have $\stabsupp(y,A) \leq \sqrt{\sigma_2(y,A)}$.
\label{lemma:sigma2fscond}
\item 
We have $\stabsupp(y,A) \leq \nmax{A}\sigma_3(y,A) $.
\label{lemma:sigma3fscond}
\end{enumerate}
\end{proposition}

Similarly, a lower bound which makes use of the following polynomial
(defined on positive $\nu,\xi$)
\begin{equation}\label{eq:q}
q(\nu,\xi) := 96\nu^5+12\nu^3(1+\lambda \sqrt{N})\sqrt{\xi}
+ \xi \left(\frac{2\nu^3}{\lambda} + 3\nu\right)
\end{equation}
is given below.

\begin{proposition}\label{proposition:rhofssigmalb}
Set $\alpha = \nTrTwo{y,A}$ and 
$\sigma = \sigma(y,A)$. Then
\[
  \stabsupp(y,A) \geq (mN)^{-\frac{1}{2}}\min\left\{\frac{\sigma^2}
  {q(\alpha,\sigma)},
 \frac{\sqrt{\sigma}}{6\alpha},\alpha \right\}.
\]	
\end{proposition}

Propositions~\ref{prop:sigma1MinimiserCompute}
and~\ref{proposition:rhofssigmaub} allow us to compute upper bounds
for the condition number $\condfsul(y,A)$, whilst
Proposition~\ref{proposition:rhofssigmalb} ensures that these
estimates are accurate. There is also another consequence of
these results: they allow us to provide an alternative definition for
the condition number.

\begin{definition}
For $(y,A)\in\real^m\times\real^{m\times N}$, $(y,A)$ is said to be
\emph{$\sigma$-ill-posed} for the UL feature selection problem
if $\sigma(y,A) = 0$. 
\end{definition}

Propositions~\ref{proposition:rhofssigmaub}
and~\ref{proposition:rhofssigmalb} then show that the set of
$\sigma$-ill-posed problems is exactly the set $\Sigmul$. Thus an alternative
definition of the condition number for $(y,A)$ is to define it as the
reciprocal of the distance to the set of $\sigma$-ill-posed problems. Whilst
$\sigma$ is convenient from a computational perspective, it is more
intuitive to define condition in terms of $\stabsupp$.

\section{Proofs of the stated results}\label{sec:proofs}

\subsection{Proof of Proposition~\ref{prop:goodSetOneSolution}}

Before proving Proposition~\ref{prop:goodSetOneSolution}, we need to
introduce the concept of the set of solutions with minimal support. 

\begin{definition}
The set of {\em solutions with minimal support} of a pair $(y,A)$ is
\begin{equation}\label{eq:def-min-supp}
\Sol^{\mathsf{ms}}(y,A):= \lbrace
x \in \mUL(y,A) \, | \, \forall
x' \in \mUL(y,A), \, \supp(x') \subseteq \supp(x) \Rightarrow
x'=x \rbrace.
\end{equation}
\end{definition}

\begin{lemma}\label{lemma:multipleSolutionsImpliesMS}
If $|\mUL(y,A)|\neq 1$ then $|\Sol^{\mathsf{ms}}(y,A)|\geq 2$.
\end{lemma}

\begin{proof}
The set $\mUL(y,A)$ is compact, convex, and  non-empty (from (UL1-2)). In
particular, by the Krein-Milman
Theorem~\cite[Theorem 9.4.6]{TopologicalVectorSpaces},
$\mUL(y,A)$ is the closed convex hull of its extreme points (that is,
points $p \in \mUL(y,A)$ so that if $p' \in \real^N$
is such that $p+p' \in \mUL(y,A)$ and $p-p' \in \mUL(y,A)$ then
$p' = 0$~\cite[Theorem 9.2.2(d)]{TopologicalVectorSpaces}). 
Therefore there must be at least two extreme points of
$\mUL(y,A)$; otherwise, $|\mUL(y,A)|=1$.

To complete the proof, we now show that every extreme point of $\mUL(y,A)$
is in $\Sol^{\mathsf{ms}}(y,A)$. Suppose otherwise, that is, that there is
an extreme point $x \in \mUL(y,A)$ and an $x' \in \mUL(y,A)$ with
$x' \neq x$ and $\supp(x') \subseteq \supp(x)$.
Then (by (UL4-5)), $Ax = Ax'$, $\sgn(x) = \sgn(x')$ and
$\|x\|_1 = \|x'\|_1$.
For $\epsilon > 0$, let $v = (1+\epsilon)x - \epsilon x'$.
If $\epsilon$ is sufficiently small we have that $\sgn(v) = \sgn(x)$
and hence
\begin{align*}
\|v\|_1 = \sum_{i \in \supp(v)} v_i\sgn(v_i) = \sum_{i \in \supp(x)} v_i\sgn(x_i)
&= \sum_{i \in \supp(x)} (1+\epsilon)|x_i| - \epsilon |x'_i|\\
&= (1+\epsilon)\|x\|_1 - \epsilon \|x'\|_1 = \|x\|_1.
\end{align*}
	
Moreover, $Av = Ax$. We thus conclude that $v \in \mUL(y,A)$. But this
contradicts the extremality of $x$ (take
$p'=\epsilon(x-x')\neq 0$; then $x+p'=v\in\mUL(y,A)$ and
$x-p'=(1-\epsilon)x+\epsilon x'\in\mUL(y,A)$ since this set is
convex) and thus all extreme points of $\mUL(y,A)$ are in
$\Sol^{\mathsf{ms}}(y,A)$. We have already shown that there
are multiple extreme points in $\mUL(y,A)$, thus completing
the proof.
\end{proof}

This immediately implies Proposition~\ref{prop:goodSetOneSolution}:
indeed, if $|\mUL(y,A)| \neq 1$ then by
Lemma~\ref{lemma:multipleSolutionsImpliesMS} there exist
$v^1,v^2 \in \mUL(y,A)$ with $\supp(v^1) \neq \supp(v^2)$. By the
definition of $\stabsupp$, we must have $\stabsupp(y,A) = 0$.

\subsection{Proof of Proposition~\ref{prop:sigma1MinimiserCompute}}
To prove part~(1), first note that the result is trivial if
$S = \varnothing$: indeed, in this case $x = 0$ and hence by 
(UL4) the solution is unique. 
We thus consider the case where $S \neq \varnothing$. Assume that
the vector $\tilde{x}$ is such that 
$\tilde{x} \in \mUL(y,A)$. Then $A\tilde{x} = Ax$ (note that this
was stated in Section~\ref{section:BasicFactsAboutUL} as~(UL4)) and so
$\|A^*_{S^{\mathsf{c}}}(A\tilde{x} - y)\|_{\infty} < \lambda/2$. By the KKT conditions 
this implies that $\supp(\tilde{x}) \subseteq S$ (note, in the case
$S^{\mathsf{c}} = \varnothing$ this is trivial). But then 
$A_S\tilde{x}_S = A\tilde{x} = Ax = A_Sx_S$. Since $A_S^*A_S$ is invertible by
assumption, it must have a trivial kernel and hence $x_S = \tilde{x}_S$.
Finally, since both $\supp(x), \supp(\tilde{x}) \subseteq S$ we must
have $x = \tilde{x}$. Thus $|\mUL(y,A)| = 1$. The result about
$\sigma_1,\sigma_2$ and $\sigma_3$ in this circumstance follows
from the fact that in this case, the infima in
Definition~\ref{def:sigma1sigma2sigma3} are taken over a single
vector with finitely many entries.

Part 2 follows immediately from the definition of $\sigma$: note that 
$\|A^*_{S^{\mathsf{c}}}(Ax-y)\|_{\infty} = \lambda/2$ implies that $\sigma_1 = 0$
and non-invertibility of $A_S^*A_S$ implies that $\sigma_2 = 0$.
\eproof

\subsection{Proof of Proposition~\ref{proposition:rhofssigmaub}}

We begin with the following Lemma.

\begin{lemma}\label{lemma:sigma20illcond}
If $\sigma_2(y,A) = 0$ then $(y,A) \in \Sigmul$.
\end{lemma}

\begin{proof}
By the definition of $\sigma_2$, if $\sigma_2(y,A) = 0$ then there is a
minimiser $x \in \mUL(y,A)$ and a set $S\neq\varnothing$
such that $\supp(x) = S$ and
$A^*_SA_S$ is non-invertible. In particular, since
$A^*_SA_S \in \real^{|S| \times |S|}$ it must also 
have a non-trivial nullspace. Let $v\in \real^{N}$ be such that
$A^*_S A_S v_S = 0$ (so that 
$\|A_S v_S \|_2^2 = \langle v_S,A_S^*A_S v_S \rangle = 0$)
and $v_{S^{\mathsf{c}}} = 0$.
For $\epsilon > 0$ sufficiently small 
we must have $\sgn(x_S+\epsilon v_S) = \sgn(x_S)$ and so 
$\|x_S\pm \epsilon v_S\|_1 = \langle \sgn(x_S),x_S \pm \epsilon v_S\rangle  = 
\|x\|_1 \pm \epsilon \langle \sgn(x_S),v_S \rangle$. Thus at least one of 
$\|x_S + \epsilon v_S\|_1, \|x_S-\epsilon v_S\|_1$ is bounded above by
$\|x\|_1$. 
Assume that $\|x_S+\epsilon v_S \|_1 \leq \|x\|_1$ (the 
argument for $\|x_S-\epsilon v_S\|_1 \leq \|x\|_1$ is identical). 
In this case, we have $A_S^*A_S (x_S+\epsilon v_S) = A_S^* A_S x_S 
= -\lambda \sgn(x_S)/2 = -\lambda \sgn(x_S+\epsilon v_S)/2$
and $\|A^*_{S^{\mathsf{c}}}A_S (x_S+\epsilon v_S)\|_{\infty} =
\|A^*_{S^{\mathsf{c}}}A_Sx_S\|_{\infty} \leq \lambda/2$. 
Therefore $(x+\epsilon v)$ obeys the KKT conditions and so 
$(x+\epsilon v ) \in \mUL(y,A)$. But then $|\mUL(y,A)| \geq 2$ and so by 
Proposition~\ref{prop:goodSetOneSolution} we have 
$(y,A) \in \Sigmul$.
\end{proof}

We can now prove Proposition~\ref{proposition:rhofssigmaub}.

%\begin{proof}[Proof of Proposition~\ref{proposition:rhofssigmaub}]
\noindent{\em Proof of Proposition~\ref{proposition:rhofssigmaub}.}\quad
All parts are trivial if $\stabsupp(y,A)=0$. We then assume that 
$\stabsupp(y,A)>0$. Proposition~\ref{prop:goodSetOneSolution} then
shows that there 
exists a unique point $x$ in $\mUL(y,A)$. Let $S:=\supp(x)$. 
We now prove each of the three parts separately. Note that, because $x$ is the
only point in $\mUL(y,A)$, we don't need to take infima in the definition
of $\sigma_1,\sigma_2$ and $\sigma_3$. Furthermore, if $S = \varnothing$
then parts (2) and (3) are trivial since
$\sigma_2(y,A) = \sigma_3(y,A) = \infty$, so we assume $S\neq\varnothing$
for parts (2) and (3).

\textbf{Proof of part~\ref{lemma:sigma1fscond}:} Note that the condition that
$\sigma_1(y,A) < \lambda/4$ implies that $A$ is not the zero matrix and so
$\|A\|_{\max} \neq 0$. It also implies that $S^{\mathsf{c}} \neq \varnothing$
since otherwise $\sigma_1(y,A) = \lambda/2$.
We will argue by contradiction and assume that the statement does not hold.
Let $t$ be a real number such that 
\begin{equation*}
 \sigma_1(y,A)< t < \min\left\{\frac{\lambda\, \stabsupp(y,A)}{4\nmax{A}},
 \frac{\lambda}{4}\right\}
\end{equation*} 
(note that this interval is non-empty by our assumption). Then  
there is an $i \in \lbrace 1,2,\dotsc,N \rbrace$ such that  
$|[A_{i}^*(Ax-y)]| > \lambda/2 - t$ and $x_i = 0$. Set
$\tilde{A} = A(\Id_N+\delta P_i)$
where $\delta = 2t/(\lambda - 2t) < 4t/\lambda$ and where $P_i$ is 
an $N \times N$ matrix consisting only of $0$s except the entry
on the $i$th row, $i$th column, which is set to $1$. Clearly 
$\|A-\tilde{A}\|_{\max} \leq \delta\|A\|_{\max}\leq
4t\|A\|_{\max}/\lambda<\stabsupp(y,A)$. 
Thus $\supp(\tilde{x}) = \supp(x)$ for all $\tilde{x} \in \mUL(y,\tilde{A})$.

We claim that $x \in \mUL(y,\tilde{A})$. Indeed, since $x_i = 0$
and $\supp(\tilde{x}) = \supp(x)$, we must have $\tilde{x}_i = 0$.
Thus, again as $x_i=0$, 
\begin{equation*}
\|\tilde{A}x-y\|^2_2 + \lambda \|x\|_1 =  \|Ax-y\|^2_2 +\lambda\|x\|_1 
\leq   \|A\tilde{x} - y\|^2_2 + \lambda\|\tilde{x}\|_1 
=  \|\tilde{A}\tilde{x} - y\|^2_2 + \lambda\|\tilde{x}\|_1,
\end{equation*} 
the inequality by the optimality of $x$ and the last equality as
$\tilde{x}_i=0$.
This shows the claim.

By the assumption that $|[A^*(Ax-y)]_i| > \lambda/2 - t$ we have 
\begin{equation*}
|\tilde{A}^*_i (\tilde{A}x-y)| = (1+\delta) |A_i^*(\tilde{A}x - y)| 
= (1+\delta)|A_i^*(Ax-y)| > \!\left(1+\frac{2t}{\lambda - 2t}\right)\!\!
\left(\frac{\lambda}{2} - t\right)\! = \lambda/2
\end{equation*}
but this contradicts the fact that
$\|\tilde{A}^*(\tilde{A}x-y)\|_{\infty} \leq \lambda/2$
by the KKT conditions for the LASSO problem and the fact that
$x \in \mUL(y,\tilde{A})$.

\textbf{Proof of part~\ref{lemma:sigma2fscond}:}
Suppose for the sake of contradiction that for some $\epsilon > 0$ we have 
$\sigma_2(y,A) = t$ but $\stabsupp(y,A) > \sqrt{t+\epsilon}$. Then 
$\|(A^*_SA_S)^{-1}\|_2 =\frac1t > 1/(t + \epsilon)$. This implies that there
is a $v \in \real^{N}$ such that $v_{S^{\mathsf{c}}} = 0$, $\|v\|_2 = 1$ and
$\|A^*_SA_Sv_S\|_2 \leq t+\epsilon$. Therefore, 
$\|A_Sv_S\|_2^2 =\langle A_S^*A_Sv_S,v_S\rangle \leq \|v_S\|_2 \|A^*_SA_Sv_S\|_2$
and we obtain $\|A_Sv_S\|_2 \leq \sqrt{t+\epsilon}$.

Let $B:=A-Avv\transp$. Then 
\[
\nmax{B - A} \leq \nTwoTwo{B-A} = \nTwoTwo{Avv\transp} 
\leq \|Av\|_2\,\|v\transp\|_2 =\|Av\|_2\le \sqrt{t+\epsilon}.
\]
Hence, since $\stabsupp(y,A)> \sqrt{t+\epsilon}$, 
$(y,B) \notin \Sigmul$ and there exists
$\hat{x} \in \mUL(y,B)$ with $\supp(\hat{x}) = S$
(this comes from the definition of $\stabsupp$). But
$$
B_Sv_S= A_Sv_S-Avv_S\transp v_S =A_Sv_S -Av\|v_S\|^2 = A_Sv_S -Av=0
$$
and hence, $B_S^*B_Sv_S = 0$. It follows that $B_S^* B_S$ 
is not invertible. This contradicts Lemma~\ref{lemma:sigma20illcond}.

\textbf{Proof of part~\ref{lemma:sigma3fscond}:}
For shorthand, let $t:=\sigma_3(y,A)$. Then there exists 
an index $i \in \lbrace 1,2,\dotsc,N \rbrace$ such that $|x_i| = t$. 
Let $\tilde{y} = y- x_iAe_i$ where $e_i$ is the $i$th vector in the
standard basis on $\real^{N}$. We claim that
$\tilde{x} = x - x_i e_i$ is such that $\tilde{x} \in \mUL(\tilde{y},A)$. 
Since the KKT conditions are both necessary and sufficient for
unconstrained LASSO, it suffices to show that $\tilde{x}$ obeys the
KKT conditions. 
Let $W$ be the support of $\tilde{x}$. Since $\tilde{x}_W = x_W$,
we have $A\tilde x = A_W\tilde x_W = A_W x_W$ so that
\begin{equation*}
A^*(A\tilde{x} - \tilde{y}) 
= A^*(A_Wx_W - \tilde y) = A^*(A_W x_W + x_i Ae_i - y) = A^*(Ax - y).
\end{equation*}
Therefore $A_W^*(A\tilde{x} - \tilde{y}) = -\lambda \sgn(x)_W/2
= -\lambda\sgn(\tilde{x})_W/2$ (since $x$ obeys the KKT conditions) and on
$W^{\mathsf{c}}$ we have
$\|A_{W^{\mathsf{c}}}^*(A\tilde{x} - \tilde{y})\|_{\infty}
= \|A_{W^{\mathsf{c}}}^*(Ax - y) \|_{\infty} \leq \lambda/2$
(again, since $x$ obeys the KKT conditions). Thus $\tilde{x}$ obeys the
KKT conditions and so $\tilde{x} \in \mUL(\tilde{y},A)$ as claimed.
But then $\supp(\tilde{x}) \neq \supp(x)$ and so
$
\stabsupp(y,A) \leq d_{\max}(\left[(y,A),(\tilde{y},A)\right])
\leq \|x_iAe_i\|_{\infty} \leq t \nmax{A}.
$
\qed

\subsection{Proof of Proposition~\ref{proposition:rhofssigmalb}}

We assume that $\sigma > 0$, otherwise there is nothing to prove. Let 
$x \in \mUL(y,A)$ and $S = \supp(x)$. Additionally, let
$\Delta: = \min\left(\frac{\sigma^2}{q(\alpha,\sigma)}, 
\frac{\sqrt{\sigma}}{6\alpha},\alpha \right)$ and
$\delta := (mN)^{-1/2}\Delta$. 
Finally, let $(\tilde{y},B)$ be such that
$\|\tilde{y} - y\|_{\infty} \leq \delta$ and $\nmax{A-B} \leq \delta$. 

We will use the following results derived from the classic
bounds (see~\cite[\S1.1]{Condition} or~\cite[\S6.2]{Higham96}) between
norms and the definition of~$\Delta$:
\begin{align}
\label{eq:med1}
\nTwoTwo{B-A} &\leq \sqrt{mN} \nmax{B-A} \leq \sqrt{mN}\delta \leq \Delta,\\
\label{eq:med2}	
\nTwoTwo{B} &\leq \nTwoTwo{A} + \Delta \leq 2\alpha,\\
\label{eq:med3}
\|\tilde{y} - y\|_2 &\leq \sqrt{m} \delta \leq \Delta,\\
\label{eq:med4}	
\|\tilde{y}\|_2 &\leq \|y\|_2 + \|\tilde{y}- y\|_2
\leq \alpha + \Delta \leq 2\alpha.
\end{align}
We start with the case where both $S$ and $S^{\mathsf{c}}$ are non empty. 
We initially establish a sequence of basic inequalities. We begin by
observing that
\begin{align*}
B_S^*B_S &= A^*_SA_S + (B_S - A_S)^*A_S + A_S^*(B_S - A_S)
+ (B_S - A_S)^*(B_S -A_S)
\end{align*}
so that $B^*_SB_S = A^*_SA_S(\Id+X)$ where
\[
  X = (A^*_SA_S)^{-1}\left[(B_S - A_S)^*A_S + A_S^*(B_S - A_S)
  + (B_S - A_S)^*(B_S -A_S)\right].
\]
Note that $X$ is well defined since $(A^*_SA_S)^{-1}$ exists by the
assumption that
$\sigma > 0$ and Proposition~\ref{prop:sigma1MinimiserCompute}.
In addition, 
\begin{align}
\nTwoTwo{X} &\leq 
\nTwoTwo{(A^*_SA_S)^{-1}}
\,(\nTwoTwo{(B_S - A_S)^*A_S} + \nTwoTwo{A_S^*(B_S - A_S)}\notag\\
&\,\,+ \nTwoTwo{(B_S - A_S)^*(B_S -A_S)})\notag\\
&\leq \nTwoTwo{(A^*_SA_S)^{-1}}
\left(2\nTwoTwo{B_S - A_S}\,\nTwoTwo{A_S} + \nTwoTwo{B_S -A_S}^2\right)\notag\\
&\le
\frac{2\Delta \nTwoTwo{A} + \Delta^2}{\sigma_2(y,A)}
\leq \frac{3\Delta \alpha}{\sqrt{\sigma}}\label{eq:XBound}
\end{align}
where we used the definition of $\sigma_2$,~\eqref{eq:med1} and  
$\Delta, \nTwoTwo{A} \leq \alpha$. It follows from the hypothesis on
$\Delta$ that $\nTwoTwo{X} < 1/2$.  
Hence, $\Id+X$ is invertible with inverse satisfying
$(\Id+X)^{-1} = \sum_{r=0}^{\infty} (-1)^rX^r$ and consequently,
so is $B_S^*B_S$ and 
we have $(B_S^*B_S)^{-1} = (\Id+X)^{-1}(A^*_SA_S)^{-1}$. Furthermore, 
\begin{align}
\nTwoTwo{(B^*_SB_S)^{-1} \!- (A_S^*A_S)^{-1}}
\!&= \!\nTwoTwo{(\Id+X)^{-1} (A_S^*A_S)^{-1} - (A_S^*A_S)^{-1}} \notag\\
&\leq\! \nTwoTwo{(\Id+X)^{-1} - \Id} \,\nTwoTwo{(A_S^*A_S)^{-1}}\notag\\
& \leq \! \nTwoTwoAl{\sum_{r=1}^{\infty} (-1)^r X^r }\!\!
\nTwoTwo{(A_S^*A_S)^{-1}} \!\leq \!
\frac{\nTwoTwo{X}\nTwoTwo{(A^*_SA_S)^{-1}}}{1-\nTwoTwo{X}} \notag \\
&<\! \frac{2\nTwoTwo{X}}{\sigma_2(y,A)}
\leq \frac{6\Delta\alpha}{\sigma},\label{eq:med5}
\end{align}
where the last inequality following from~\eqref{eq:XBound}.  In particular,
\begin{align}\label{eq:med6}
& \big\|(B_S^*B_S)^{-1}B^*_S\tilde{y} - (A_S^*A_S)^{-1}A^*_Sy\big\|_2\notag\\ 
 \leq\;& \big\|\big[(B_S^*B_S)^{-1}B^*_S - (A_S^*A_S)^{-1}B_S^*\big]
 \tilde{y}\big\|_2 + \big\|\big[(A_S^*A_S)^{-1}B^*_S - (A_S^*A_S)^{-1}A^*_S\big]
 \tilde{y}\big\|_2 \\
+&\big\|(A_S^*A_S)^{-1}A_S^*(y - \tilde{y})\big\|_2 \notag\\
 \leq\;& \nTwoTwo{(B^*_SB_S)^{-1} - (A_S^*A_S)^{-1}} \nTwoTwo{B^*_S}
 \|\tilde{y}\|_2 \\
\;+&\nTwoTwo{(A_S^*A_S)^{-1}} \big(\nTwoTwo{B^*_S - A^*_S} \|\tilde{y}\|_2 
+ \nTwoTwo{A^*_S} \|y- \tilde y\|_2\big)\notag\\
\leq\;& \frac{6\Delta\alpha(2\alpha)(2\alpha)}{\sigma}+ 
\frac{ \Delta(2\alpha)+\Delta \alpha}{\sigma_2(y,A)}
\qquad\mbox{by~\eqref{eq:med1},~\eqref{eq:med5},~\eqref{eq:med3},
	and ~\eqref{eq:med4}}\notag\\
=\;& \frac{24\Delta\alpha^3}{\sigma}+\frac{3\Delta \alpha}{\sqrt{\sigma}}
=: \Upsilon.
\end{align}	

Since $\sigma>0$ and $x$ obeys the KKT conditions for unconstrained LASSO
with input $(y,A)$, we see that
$x_S = (A_S^* A_S)^{-1}(A_S^* y  - \lambda\sgn(x)_S/2 )$ and $x_{S^{\mathsf{c}}} =0$.
Hence, if we let $\tilde{x}\in\real^N$ be given by 
$\tilde{x}_S := (B_S^*B_S)^{-1}(B_S^*\tilde{y} - \lambda\,\sgn(x)_S/2)$ and 
$\tilde{x}_{S^{\mathsf{c}}} := 0$ then 
\begin{align}\label{eq:med7}
2\|\tilde{x} - x\|_2 
&\leq \|(B^*_S B_S)^{-1} (2B_S^*\tilde{y} - \lambda\,\sgn(x)_S) - (A^*_S A_S)^{-1} 
 (2A_S^*y - \lambda\,\sgn(x)_S)\|_2\notag\\
&\leq  2\left\|(B^*_S B_S)^{-1} B_S^*\tilde{y}
  - (A^*_S A_S)^{-1}A_S^*y\right\|_2  \\
&+\lambda\left\|\left[(B^*_S B_S)^{-1} - (A^*_S A_S)^{-1} \right]
 \sgn(x)_S\right\|_2\notag\\
&\!\!\underset{\eqref{eq:med6}}{\leq} 2\Upsilon
 + \lambda\nTwoTwo{(B^*_S B_S)^{-1} - (A^*_S A_S)^{-1}}\|\sgn(x)_S\|_2 
\underset{\eqref{eq:med5}}{\leq} 2\Upsilon
+ \frac{6\Delta \alpha\lambda \sqrt{|S|}}{\sigma}\notag\\
& \leq 6\Delta\alpha \left( \frac{8\alpha^2}{\sigma}
+\frac{1 + \lambda \sqrt{|S|}}{\sqrt{\sigma}}\right)
\end{align}
Also, since $x \in \mUL(y,A)$ we have (by (UL1))
$\|Ax-y\|^2_2,\lambda\|x\|_1 \leq \|y\|^2_2$. Thus 
\begin{align}\label{eq:med8}
\|B^*_{S^{\mathsf{c}}}&(B\tilde{x} -\tilde{y})
 - A^*_{S^{\mathsf{c}}}(Ax - y)\|_2 \notag\\
\leq& \|B^*_{S^{\mathsf{c}}}(B\tilde{x} - \tilde{y})
 - B^*_{S^{\mathsf{c}}}(Bx - \tilde{y})\|_2 
 + \|B^*_{S^{\mathsf{c}}}(Bx - \tilde{y})
 - B^*_{S^{\mathsf{c}}}(Ax - \tilde{y})\|_2\notag\\
&+\|B^*_{S^{\mathsf{c}}}(Ax - \tilde{y})- B^*_{S^{\mathsf{c}}}(Ax - y)\|_2\ %notag\\
 + \|B^*_{S^{\mathsf{c}}}(Ax - y)- A^*_{S^{\mathsf{c}}}(Ax - y)\|_2\notag\\
\leq& \nTwoTwo{B}\left(\nTwoTwo{B} \|\tilde{x} - x\|_2 + \nTwoTwo{B-A}\|x\|_2
 + \|y-\tilde{y}\|_2\right) + \nTwoTwo{B-A} \|Ax-y\|_2\notag\\
\leq& 2\alpha\left[ 2\alpha\|(\tilde{x}-x)\|_2
 + \frac{\nTwoTwo{B-A} \|y\|^2_2}{\lambda} + \|y-\tilde{y}\|_2\right] 
 + \Delta \|y\|_2 \notag\\
\leq& 2\alpha\left[ 6\Delta\alpha^2 \left( \frac{8\alpha^2}{\sigma}+\frac{1 
 + \lambda \sqrt{|S|}}{\sqrt{\sigma}}\right) + \frac{\Delta\alpha^2}{\lambda}
 + \Delta\right] + \Delta \alpha \notag\\
=\quad&   12\Delta\alpha^3 \left( \frac{8\alpha^2}{\sigma}
 +\frac{1 + \lambda \sqrt{|S|}}{\sqrt{\sigma}}\right) 
 + \frac{2\Delta\alpha^3}{\lambda} + 3\Delta\alpha\notag\\
\leq & \Delta \left( \frac{96\alpha^5}{\sigma}+\frac{12\alpha^3 
 + 12\alpha^3\lambda \sqrt{|S|}}{\sqrt{\sigma}}
 + \frac{2\alpha^3}{\lambda} + 3\alpha\right)
\leq \Delta q(\alpha,\sigma)/\sigma.
\end{align}

We can now conclude the proof. Using~\eqref{eq:med7} and the
definition of $\Delta$ we obtain that 
\begin{eqnarray*}
\|\tilde{x} - x\|_{\infty}
&\leq& \|\tilde{x} - x\|_2 \;\leq\; 3\Delta\alpha
\left(\frac{8\alpha^2}{\sigma}+\frac{1+\lambda\sqrt{|S|}}{\sqrt{\sigma}}\right)
\;\leq\; 3\frac{\sigma^2}{q(\alpha,\sigma)}\alpha
\left(\frac{8\alpha^2}{\sigma}
+\frac{1+\lambda \sqrt{|S|}}{\sqrt{\sigma}}\right)\\
&=& \sigma\left(\frac{24\alpha^3+3\alpha\sqrt{\sigma}(1+\lambda\sqrt{|S|})}
{96\alpha^5+12\alpha^3(1+\lambda\sqrt{N})\sqrt{\sigma}+
\left(\frac{2\alpha^3}{\lambda}+3\alpha\right)}\right)\:<\; \sigma
\end{eqnarray*} 
and therefore, if we set $S^+ = \lbrace i \, \vert \, x_i > 0\rbrace$,
then for each $i \in S^+$ we have that 
$
\tilde{x}_i \geq x_i - |\tilde{x}_i - x_i| > \sigma_3(y,A) - \sigma \geq 0.
$ 
Similarly, if we set $S^- = \lbrace i \, \vert \, x_i < 0\rbrace$ then
for each $i\in S^-$ we have 
$\tilde{x}_i  < 0$. It follows that each entry of $\tilde{x}_S$ is non-zero
and that 
$\sgn(\tilde{x})_S = \sgn(x)_S$ and so, 
\begin{equation}\label{eq:med9}
2B_S^*(B\tilde{x} - \tilde{y}) = \lambda \,\sgn(x)_S
= \lambda\, \sgn(\tilde{x})_S.
\end{equation}
Using~\eqref{eq:med8} and the definition of $\Delta$ we show as above that 
$\|B^*_{S^{\mathsf{c}}}(B\tilde{x} - \tilde{y}) - A^*_{S^{\mathsf{c}}}(Ax - y)\|_2
< \sigma$
and thus 
\begin{align}\label{eq:med10}
\|B^*_{S^{\mathsf{c}}}(B\tilde{x} - \tilde{y})\|_{\infty}
&\leq \|B^*_{S^{\mathsf{c}}}(B\tilde{x} - \tilde{y}) - 
A^*_{S^{\mathsf{c}}}(Ax - y)\|_{\infty} + \|A^*_{S^{\mathsf{c}}}(Ax-y)\|_{\infty} \notag\\
&< \sigma +\lambda/2 - \sigma_1(y,A) \leq \lambda/2.
\end{align}	

Inequalities~\eqref{eq:med9} and~\eqref{eq:med10}, together with the
fact that each entry of $\tilde{x}_S$ is non-zero show that
$\tilde{x}$ satisfies the unconstrained LASSO KKT conditions
for $(\tilde{y},B)$ and that the hypothesis of
Proposition~\ref{prop:sigma1MinimiserCompute}
part~\eqref{prop:sigma1MinimiserSgNZ}
holds. We therefore have $\lbrace \tilde{x} \rbrace
= \mUL(\tilde{y},B)$. Since each entry of $\tilde{x}_S$ is non-zero
and $\tilde{x}_{S^{\mathsf{c}}}=0$, we have $\supp(\tilde{x}) = \supp(x)$ and
therefore all vectors in $\mUL(\tilde{y},B)$ have the same support as
$x$. 

We now argue that the same result holds for $S = \varnothing$ or
$S^{\mathsf{c}} = \varnothing$. In the former case, we set
$\tilde x =0$. Then the bound in~\eqref{eq:med7} holds trivially
and~\eqref{eq:med8} and ~\eqref{eq:med10} follow as before. Thus
$\tilde x$ satisfies the unconstrained LASSO KKT conditions for
$(\tilde{y},B)$ and we are done. In the later case where
$S^c= \varnothing$, the only difference is that there is no need to
compute~\eqref{eq:med8}: it suffices to use~\eqref{eq:med7} to
conclude~\eqref{eq:med9}, which shows that $\tilde x$ satisfies the
unconstrained LASSO KKT conditions.

We have thus shown in each case for $S$ that all vectors in
$\mUL(\tilde{y},B)$ have the same support as $x$, for every pair
$(\tilde{y},B)$ with $\|\tilde{y} - y\|_{\infty} \leq \delta$ and
$\nmax{A-B} \leq \delta$. We deduce that $\stabsupp(y,A) \geq \delta$
which completes the proof.

\subsection{A convex quadratic routine for unconstrained LASSO}
\label{section:subroutines}

In addition to the subroutines $\orFindMatrix^U$ and
$\orFindVector^b$, which are assumed within our computational model we
will use the following subroutine:
\begin{quote}
\begin{description}
\item[$\orULasso(y,A,\lambda)$:] given $y\in{\mathbb Q}^m$,
$A\in{\mathbb Q}^{m\times N}$, and $\lambda\in{\mathbb Q}$, it 
returns a vector $x \in \mUL(y,A)$ in time $O(N^3 \log_2(L))$ where
$L$ is the total number of bits of $A$, $y$ and $\lambda$.
\end{description}
\end{quote}
The existence of $\orULasso$ follows from the fact that unconstrained
LASSO can be written as a convex quadratic program. To see this, if we write
$x = x^+ - x^-$ and 
$\tilde{x} = \begin{pmatrix}x^+ & x^-\end{pmatrix}\transp$ with 
$x^+ \geq 0$ and $x^- \geq 0$ (where these inequalities are taken
entrywise) we have
\begin{align*}
\|Ax-y\|^2_2 + \lambda \|x\|_1
&= \langle Ax,Ax \rangle - 2\langle A^*y,x \rangle 
 + \lambda \sum_{i=1}^{N} (x^+_i + x^-_i) + \|y\|^2_2\\
&= \tilde{x}^* M\tilde{x} + (\lambda\mathbf{1}_{2N}-2By)^* \tilde{x} + \|y\|^2_2 
\end{align*}
where $\mathbf{1}_N$ is a vector of length $N$ with each entry equal to $1$, 
$M = \begin{pmatrix} A^*A &-A^*A\\ -A^*A & A^*A \end{pmatrix}$ and 
$B = \begin{pmatrix} A^* \\ - A^* \end{pmatrix}$. Note also that $M$
is positive semi-definite. Since $\|y\|^2_2$ is constant, we conclude
that the solutions to the quadratic program in standard form 
\[
 \argmin_{\tilde{x} \in \real^{N}} \{\tilde{x}^* M\tilde{x} 
 + (\lambda\mathbf{1}_{2N}-2By)^* \tilde{x} \, \vert \, -\Id_{2N}\tilde{x}
 \leq 0\rbrace
\] 
can be converted to solutions of the unconstrained LASSO problem with
inputs $A$ and $y$. Thus using the algorithm proposed in
e.g.~\cite{MonteiroON3L}, we obtain an algorithm that works in 
$\Oh(N^3L) = \Oh(N^{3}p + N^{3}\log_2(m+N)) = \Oh(N^{3}(p+\log_2(N)))$
arithmetic operations, where $L$ is the total number of bits required
to store $A$, $y$ and $\lambda$ and $p$ is the maximal number of bits
required to store any entry of $A$, $y$ and $\lambda$. 

\subsection{A subroutine for testing upper bounds for $\sigma$}
\label{section:UpperBoundSubroutine}
Our aim in this section is to produce an algorithm that can tell us,
for a given value of $C$, whether or not $\sigma \leq C^2$. 
More precisely, we aim to produce the following subroutine to add to the
one discussed above,
\begin{quote}
\begin{description}
\item[$\orSigma(y,A,x,\lambda,C)$:] with input $y \in \mathbb{Q}^m$, 
$A\in \mathbb{Q}^{m \times N}$, $x\in \mathbb{Q}^N$, and $\lambda,C\in\mathbb{Q}$, 
returns $\mathsf{true}$ if $\sigma(y,A) \leq C^2$ and $\mathsf{false}$
otherwise. The precondition is that $x \in \mUL(y,A)$ and $\lambda$ is
the unconstrained LASSO parameter. The cost of running this procedure 
is $\Oh(N^3)$.
\end{description}
\end{quote}

Executing $\orSigma(y,A,x,\lambda,C)$ requires deciding if
$\sigma_1(y,A) \leq C^2$, $\sigma_2(y,A) \leq C$ and
$\sigma_3(y,A) \leq C^2$. Let
$S:=\supp(x)$.  Note that computing $\sigma_1$ can be done by
computing $A^*_{S^{\mathsf{c}}}(Ax - y)$ (at a cost of $\Oh(mN)$
operations), finding the maximum absolute value across all rows
(taking $\Oh(|S^{\mathsf{c}}|) = \Oh(N)$ operations) and then subtracting
$\lambda/2$. Hence $\sigma_1$ can be
computed in $\Oh(mN)$ operations. We compute $\sigma_3$ via a simple maximum
argument requiring $\Oh(N)$ operations. The hardest of the three to
compute is $\sigma_2$; we will approximate the smallest singular value
of $A^*_S A_S$ by testing the positive definiteness of $A^*_S A_S -
t \Id$ for various values of $t$. Note that if $A^*_S A_S - t\Id$ is
not positive definite then $\nTwoTwo{(A_S^*A_S)^{-1}}^{-1} \leq
t$. Conversely, if $A^*_SA_S - t\Id$ is postiive definite then
$\nTwoTwo{(A^*_SA_S)^{-1}}^{-1} > t$. Indeed, the following chain of
equivalences hold since $A^*_SA_S$ is symmetric: $A^*_SA_S-
t\Id$ is positive definite $\iff$ each eigenvalue of $A^*_SA_S$ is
larger than $t$ $\iff$ each eigenvalue of $(A^*_SA_S)^{-1}$ is smaller
than $t^{-1}$ $\iff \nTwoTwo{(A^*_SA_S)^{-1}} < t^{-1}$.

\begin{remark}
It is well known that there is an algorithm $\orPosDef$, that can check
if a symmetric $r \times r$ matrix is positive definite or not.
Such an algorithm runs with $\mathcal{O}(r^3)$ operations. 
\end{remark}

We can now precisely state our algorithm for computing $\sigma$
(where we take by convention $\|A^*_{\varnothing}(Ax-y)\|_{\infty} = 0$
and $A^*_{\varnothing}A_{\varnothing}-C\Id$ as positive definite for any $C>0$).

\SetKw{ands}{and}
\begin{algorithm2e}
\SetAlgoLined
\TitleOfAlgo{$\orSigma(y,A,x,\lambda,C)$}
\KwData{$A\in \real^{m \times N}$, $y \in \real^{m}$, 
	$x \in \mUL(y,A)\subseteq\real^{N}$, $C > 0$, and the 
	LASSO parameter $\lambda$.}
\KwResult{$\mathsf{true}$ if $\sigma \leq C^2$, otherwise, $\mathsf{false}$}
$S:=\supp(x)$\;
\eIf{$\|A^*_{S^{\mathsf{c}}}(Ax-y)\|_{\infty} <\lambda/2$ \ands $A^*_S A_S$
{\rm is invertible}}
{$\sigma_1:= \lambda/2-\max_{i \in S^{\mathsf{c}}}(|A^*_{i} (Ax - y)|) $\;
$\sigma_3:= \|x_{S^{\mathsf{c}}}\|_{\infty}$\;
$X := A^*_SA_S - C\Id$\;
\eIf{$\orPosDef(X)$ \ands $\sigma_1 \leq C^2$ \ands $\sigma_3 \leq C^2$}
{
  \Return $\mathsf{true}$;
}
{
  \Return $\mathsf{false}$;
}
}
{\Return $\mathsf{true}$}
\end{algorithm2e}

\begin{proposition}\label{prop:Sigma}
\sloppy Algorithm $\orSigma$ is correct (the returned value is
$\mathsf{true}$ iff $\sigma(y,A) \leq C^2$).
Its running time is bounded by $\Oh(N^3)$.
\end{proposition} 

\begin{proof}
We begin proving correctness. Because $x\in\mUL(y,A)$, 
by Proposition~\ref{prop:sigma1MinimiserCompute}(1), 
if $A^*_{S^{\mathsf{c}}}(Ax-y) < \lambda/2$ and $A^*_S A_S$ is invertible, then
$x$ is the only solution of $(y,A)$. In this case the infima in the
definition of $\sigma_1,\sigma_2$ and $\sigma_3$ reduces to the
corresponding values at $S$ and the correctness of the estimates for the
$\sigma$s has already been argued above. If, instead,
either $A^*_{S^{\mathsf{c}}}(Ax-y) \ge \lambda/2$ or $A^*_S A_S$ is invertible,
then, by Proposition~\ref{prop:sigma1MinimiserCompute}(2), 
$\sigma = 0$ and, clearly, $\sigma\le C^2$. 

We now show the complexity bound. 
As mentioned earlier, computing $\sigma_1$ takes $\Oh(mN)$ operations 
and computing $\sigma_3$ takes $\Oh(N)$ operations.  
Computing $X$ can be done using $\Oh(|S|^2N) = \Oh(N^3)$ operations.
Also, as discussed earlier, the computation of $\orPosDef$
takes $|S|^3/3 = \Oh(N^3)$ operations. 
Thus the total complexity of the algorithm is
$\Oh(N^3 + mN + N + N^3) = \Oh(N^3)$.
\end{proof}

\subsection{Proof of Theorem~\ref{thm:fsulcomp}, parts~(1) \& (2)}
\label{section:ProofOfMainTheorem}
Our argument for both parts (1) and (2) will involve increasing the
precision of the approximations $(y,A)$ of the true
input, which we denote throughout the proof by $(b,U)$, until 
$\sigma(y,A)$ is sufficiently large. This strategy will give us both the
solution to unconstrained LASSO feature selection (part 1) and an upper
bound for the condition number (part 2). This is done in the algorithm
$\orFSUL$.

\begin{algorithm2e}
\SetAlgoLined
\TitleOfAlgo{$\orFSUL$}
\vspace{-0.15cm}
\KwData{$\lambda>0$ as well as oracles $\orFindMatrix^U$ and 
	$\orFindVector^b$ that approximate $(b,U)$ to 
	arbitrary precision}
\KwResult{a support set $S\subseteq\{1,\ldots,N\}$ and $\eta$ such that
$\Cond(b,U) \leq \eta$}
$\delta := 1, \quad n:=0,\quad \delta^{1/4}:=1$\;
\Repeat{{\bf not} $\orSigma(y,A,x,\lambda,C)$ {\bf and}
$G^2 (mN)^{-1}  \geq 4\delta^2$}{
	$n:=n+1$\;
	$\delta:=\delta/16$\;
	$A:= \orFindMatrix^U(4n)$\;
	$y:= \orFindVector^b(4n)$\;
	$x:= \orULasso(y,A,\lambda)$\;
	$\delta^{1/4}:=\delta^{1/4}/2$\;
	$G:=\nTrMax{y,A}$\;
	$H:=\nTrOneStar{y,A}$\;
	$C:=6 \cdot \delta^{1/4}N(\lambda + \lambda^{-1})H^2$\;
}
\Return $S := \supp(x)$ and $\eta = \delta^{-1}$.
\end{algorithm2e}

There are two things that need to be shown: firstly, that the
algorithm is correct in the sense that if the {\bf repeat} loop
terminates, the output $S$ is correct and $\eta$ is a bound for
$\condfsul(b,U)$. Secondly, that the runtime is bounded as stated,
which will also give us a guarantee that the {\bf repeat} loop does in
fact terminate when $\condfsul(b,U)<\infty$.  The correctness proof
will require the use of Proposition~\ref{proposition:rhofssigmalb} and
the runtime bound the use of
Proposition~\ref{proposition:rhofssigmaub}. We split the proof of the
runtime into two further sections: firstly, we evaluate the cost of
each iteration of the loop, and secondly we show a bound on the number
of iterations that are executed.

\subsubsection{Proof of correctness:}
\sloppy Assume that the repeat loop terminates. We use the notation
$S,n,\delta,A,y,x,G,H$
and $C$ to denote the values of these variables set by
the algorithm when the {\bf repeat} loop has terminated.  
We claim that $\stabsupp(y,A) \geq 2\delta$. Assume for now 
that this is the case.
Then by the properties~\eqref{eq:OracleGuarantee} of
$\orFindMatrix^U$ and $\orFindVector^b$, at the $n$th
iteration of the repeat loop we must have
$d_{\infty}(b,y),d_{\max}(A,U) \leq 2^{-4n}= 16^{-n} = \delta$.
%where $d_{\infty}$ and $d_{\max}$ are the distances induced by the infinity
%and entrywise max norms, respectively.

Thus when the {\bf repeat} loop terminates, 
$\stabsupp(y,A) \geq 2\delta > d_{\max}\left[(y,A),(b,U)\right]$ and so
if $w \in \mUL(b,U)$ and $v \in \mUL(y,A)$ then $\supp(w) = \supp(v)$.
But since $x \in \mUL(y,A)$ by the correctness of $\orULasso$, we must
have $\supp(w) = \supp(x) = S$. Furthermore, since
$\stabsupp(b,U) \geq \stabsupp(y,A) - d_{\max}\left[(y,A),(b,U)\right]
\geq 2\delta - \delta = \delta$, we must have
$\condfsul(b,U) \leq \delta^{-1} = \eta$.

Therefore we have shown that if the algorithm terminates 
then it terminates with the correct result, provided that we can show that 
$\stabsupp(y,A) \geq 2\delta$. To that end, we will make use of the
following lemma.

\begin{lemma}\label{lemma:CorrectnessOfAlgorithm}
Let $H\geq \nu\ge 1$, $\delta\le 1$, and $\lambda>0$. Let 
$C = 6 \delta^{1/4}N(\lambda + \lambda^{-1})H^2$. For all $\xi \geq C^2$
we have
\begin{equation*}
 f(\xi):= \xi^2 -2\delta \sqrt{mN}q(\nu,\xi)\geq 0 \text{\quad and\quad }
 g(\xi):= \frac{\sqrt{\xi}(mN)^{-1/2}}{6\nu} - 2\delta \geq 0
\end{equation*} 
where, we recall, $q(\nu,\xi)$ is defined in~\eqref{eq:q}.
\end{lemma}

\sloppy Assuming for now that Lemma~\ref{lemma:CorrectnessOfAlgorithm} holds,
we will complete the proof that $\stabsupp(y,A) \geq 2\delta$. Indeed, when
the algorithm terminates $\sigma(y,A)> C^2$ because
$\orSigma(y,A,x,\lambda,C)$ does not hold. Now we apply 
Lemma~\ref{lemma:CorrectnessOfAlgorithm} with
$H=\nTrOneStar{y,A}$, $\alpha = \nTrTwo{y,A}$
and $\sigma=\sigma(y,A)$.
Using the inequality for $f$ in this lemma with $\xi=\sigma$
and $\nu=\alpha$ we obtain
$\sigma^2 \geq 2\delta\sqrt{mN}q(\alpha,\sigma)$ and thus
$(mN)^{-1/2}\sigma^2/q(\alpha,\sigma)\geq 2\delta$. Next, using the inequality 
for $g$, we obtain
$(mN)^{-1/2} \sqrt{\sigma}/(6\alpha)\geq 2\delta$. Finally, when the
{\bf repeat} loop terminates we must have
$G^2 (mN)^{-1}\geq 4\delta^2$ and in particular, since $\alpha \geq G$,
we get $(mN)^{-1/2} \alpha\geq2\delta$. 
Combining these inequalities yields
\begin{equation}\label{eq:CorrectnessFinalBound}
\stabsupp(y,A) \geq (mN)^{-1/2}\min(\sigma^2/q(\alpha,\sigma),
\sqrt{\sigma}/(6\alpha),\alpha)\geq2 \delta
\end{equation}
where the first inequality follows from
Proposition~\ref{proposition:rhofssigmalb}.

All that remains is to prove Lemma~\ref{lemma:CorrectnessOfAlgorithm}.
This is done as follows.

\begin{proof}[Proof of Lemma~\ref{lemma:CorrectnessOfAlgorithm}]
We make frequent use of the inequality $H \geq \nu$. 
We start by proving that $f(\xi)$ is increasing on $[C^2,\infty)$. We have
$f'(\xi) = 2\xi - 2\delta \sqrt{mN}
[6\nu^3(1+\lambda \sqrt{N})/\sqrt{\xi} + 2\nu^3/\lambda + 3]$.
The form of $f'(\xi)$ makes it clear to see that for positive $\xi$,
$f'(\xi)$ is increasing and so to show that $f(\xi)$ is increasing on
$[C^2,\infty)$ it suffices to show that $f'(C^2)\geq 0$. 
Because $\delta \leq 1$, we have $\delta^{2/4} \geq \delta^{3/4}$.
Furthermore, $H \geq \nu \geq 1$ and $N \geq 1$. Thus
\begin{equation} \label{eq:CorrectnessSecondDerBound}
\delta^{2/4} N^2(\lambda^2 + 2 + \lambda^{-2})H^4 \geq 2\delta^{2/4} N^2H
\geq \delta^{2/4} \left(1+\frac{\lambda\sqrt{N}}
{\lambda + \lambda^{-1}}\right)\nu
\geq \frac{\nu\delta^{3/4}(1+\lambda\sqrt{N})}{\lambda + \lambda^{-1}}
\end{equation}

Furthermore, $(\lambda^2 + 2 + \lambda^{-2})\lambda > 2\lambda
+ \lambda^{-1} > \lambda + \lambda^{-1} \geq 2$ by the AM-GM inequality
and the assumption that $\lambda > 0$. Thus
\begin{equation}\label{eq:CorrectnessThirdDerBound}
  \delta^{2/4}N^2(\lambda^2 + 2 + \lambda^{-2})H^4  \geq
  2\delta^{2/4}N^2\nu^3\lambda^{-1} \geq 2\nu^3\delta N \lambda^{-1}.
\end{equation}
Our final simple inequality is the following: using
$\lambda^2 + \lambda^{-2} \geq 2$
and $H,N \geq 1\geq \delta$, we get
\begin{equation}\label{eq:CorrectnessForthDerBound}
	\delta^{2/4}N^2(\lambda^2 + 2 + \lambda^{-2})H^4 \geq
	4\delta^{2/4}N^2\nu  \geq 3\delta N\nu.
\end{equation} 
Using the definitions of $f$, $C$ and the bounds $H \geq \nu$, $H,N\ge 1$,
and $m \leq N$, we get
\begin{align*}
f'(C^2) &= 2C^2 - \delta \sqrt{mN}
\left[\frac{6\nu^3(1+\lambda \sqrt{N})}{C} + \frac{2\nu^3}{\lambda}
+ 3\nu\right]\\
&\geq (2\cdot 6^2) \delta^{2/4}N^2(\lambda^2 + 2 +  \lambda^{-2})H^4 - 2\delta N
 \left[\frac{6\nu^3(1+\lambda \sqrt{N})}
 {6 \delta^{1/4}N(\lambda + \lambda^{-1})H^2}
 + \frac{2\nu^3}{\lambda} + 3\nu\right]\\
&\geq 6\delta^{2/4}N^2(\lambda^2 + 2 +  \lambda^{-2})H^4 -
 2\left[\frac{\nu\delta^{3/4}(1+\lambda \sqrt{N})}{(\lambda + \lambda^{-1})}
 + \frac{2\nu^3\delta N}{\lambda} + 3\nu\delta N\right] \geq 0,
\end{align*}
where the subtraction in  the last line decomposes into three subtractions
each of them being non-negative by
(\ref{eq:CorrectnessSecondDerBound}--\ref{eq:CorrectnessForthDerBound}).
We conclude that $f$ is increasing on $[C^2,\infty)$. Thus to show that
$f$ is positive it is sufficient to show that $f(C^2) > 0$. Indeed,
\begin{align*}
\frac{f(C^2)}{(NH^2)^4\delta} &= (6(\lambda + \lambda^{-1}))^4
 -  \frac{2\sqrt{mN}}{(NH^2)^4}\biggl[96\nu^5 + 12\nu^3
(1+\lambda \sqrt{N})(6\delta^{1/4}N(\lambda + \lambda^{-1})H^2) \\
& \quad \quad + 6^2 \delta^{2/4} N^2(\lambda + \lambda^{-1})^2 H^4
\left(\frac{2\nu^3}{\lambda} + 3\nu\right) \biggr]\\
&\geq  6^4(\lambda + \lambda^{-1})^4-2 \!\left(\!96 + 12 (1+\lambda)
(6(\lambda + \lambda^{-1}))
+6^2 (\lambda + \lambda^{-1})^2 \left(\frac{2}{\lambda}+3\right)\!\right)\\
&=  [6^4 (\lambda + \lambda^{-1})^4 -
2(180\lambda^2 + 144\lambda  + 384\lambda^0 + 216\lambda ^{-1}
+ 108\lambda^{-2} + 72\lambda^{-3})]\\
&\geq 6^4 [\lambda^4 +4\lambda^{2} + 6 + 4\lambda ^{-2}
+ \lambda^{-4} - (\lambda^2 +\lambda +1 +\lambda^{-1}
+\lambda^{-2} + \lambda^{-3})]\\
&\geq 6^4[\lambda^4 + 4\lambda^2 + 6 + 4 \lambda^{-2} + \lambda^{-4} -
(2\lambda^2 +3 + 3\lambda^{-2} + \lambda^{-4})] \geq 0.
\end{align*}
where in the last line we have made use of the facts that
$\lambda < \lambda^2 +1$, $\lambda^{-1} <\lambda^{-2} + 1$ and
$\lambda^{-3} < \lambda^{-4} +\lambda^{-2}$ all of which following from
$\lambda>0$.
Since $f$ is increasing on $[C^2,\infty)$ and
$f(C^2) \geq0$, we get that $f(\xi) \geq 0$ for $\xi \geq C^2$, completing
the proof for $f$.

We finish the argument by showing that $g(\xi)> 0$. This is
somewhat simpler: since $\xi \ge C^2$ by assumption, we have
$\xi\ge 36\delta^{2/4} N^2 (\lambda + \lambda^{-1})^2 H^4 \geq 36 \cdot 4
\delta^2 (mN)\nu^2$ since $\delta \leq 1$, $H \geq \nu$,
$N\geq m$ and $\lambda + \lambda^{-1} \geq 2$. Taking square roots yields
$\sqrt{\xi} \geq 12 \delta\nu (mN)^{1/2}$, or, equivalently,
$\sqrt{\xi} (mN)^{-1/2}/(6\nu) \geq 2\delta$.
\end{proof}

\subsubsection{A bound on the number of iterations \& precision that the repeat loop requires}
\newcommand{\rt}{\mathop{\stabsupp(b,U)}} 
\newcommand{\rs}{\triple{b,U}_{\max}}
\newcommand{\at}{\triple{y,A}_{\max}}
\newcommand{\as}{\triple{y,A}_{1}}
\newcommand{\bt}{\mathop{\stabsupp(y,A)}}

The number of iterations \& the maximum precision of the {\bf repeat} loop is bounded in the
following lemma.

\begin{lemma}\label{lemma:boundOnIterations}
There exists a universal constant $D$ independent of all parameters
such that for any $n \in \mathbb{N}$ with
\begin{equation}\label{eq:nBound}
	n \geq D \log_2\left(\max\{\lambda
        + \lambda^{-1},N,\nTrMax{b,U},\condfsul(b,U)\}\right)
\end{equation}
and any $(y,A)$ and $(b,U)$ such that
$d_{\max}[(y,A),(b,U)]\leq 2^{-n} =: \delta$, we have  
$\triple{y,A}_{\max} (mN)^{-1/2} \geq \delta$ and $\sigma(y,A) \geq C^2$ 
where $C = 6 \delta^{1/4}N(\lambda + \lambda^{-1})\nTrOneStar{y,A}^2$.

As a consequence, the number of iterations of the {\bf repeat} loop
in $\orFSUL$ (and thus, the maximum digits of precision used by $\orFSUL$) on input $(b,U)$ is bounded by 
\[
\Oh(\left\lceil \log_2\left(\max\{\lambda + \lambda^{-1},N,\nTrOneStar{b,U},
\condfsul(b,U)\}\right)\right\rceil).
\]	
\end{lemma}

\begin{proof}
As we don't need to give a precise value for $D$ in the bound on the
number of iterations we will give a shorter proof without
making this value explicit using big $\Oh$ and big $\Omega$
notation~\cite{Knuth76}.  To show that $\sigma(y,A) \geq C^2$
we need to show that $\sigma_1(y,A) \geq C^2$,
$\sigma_2(y,A) \geq C$ and $\sigma_3(y,A) \geq C^2$. We start
by noting that using well-known norm
inequalities~\cite[\S6.2]{Higham96} and the fact that $m \leq
N$ we get $\nTrOneStar{y,A} \leq N^2\nTrMax{y,A}$. Moreover,
$\nTrMax{y,A} \leq \nTrMax{b,U} + \delta$ and since
$\delta \leq 1 \leq \nTrMax{y,A}$ we obtain $\nTrMax{y,A}\leq
2\nTrMax{b,U}$. Hence
\begin{equation}\label{eq:boundnorms}
\nTrOneStar{y,A} \leq 2N^2\nTrMax{b,U}.
\end{equation}

Since $\lambda + \lambda^{-1} \geq 2$, the bound~\eqref{eq:nBound}
implies that $n \geq D\max\{2,\log_2\condfsul(b,U)\}\geq
D\max\{1,\log_2\condfsul(b,U)\}$. Hence, since 
$\bt \geq \rt - \delta$, we obtain
\begin{equation}\label{eq:btVSrt}
  \bt \geq  \rt - \delta \geq\rt -2^{-D\max\{1,\log_2\condfsul(b,U)\}}.
\end{equation}
We now claim that for $D> 2$ we have
\begin{equation}\label{eq:twoconditions}
  \condfsul(y,A) \le 2\condfsul(b,U).
\end{equation}
Indeed, if $\condfsul(b,U)\ge 2$ we see (recall, $\condfsul(b,U)=\rt^{-1}$)
that 
\[
  \rt-2^{-D\max\{1,\log_2\condfsul(b,U)\}} \ge \rt - {\rt}^D \geq \rt /2
\]
whereas if $\condfsul(b,U)\le 2$ we have (since in this case
$2^{-D} \leq 1/4 \leq \rt/2$)
\[
  \rt -2^{-D\max\{1,\log_2\condfsul(b,U)\}} = \rt - 2^{-D} \geq \rt /2.
\]
In both cases, $\rt-2^{-D\max\{1,\log_2\condfsul(b,U)\}}\geq \rt /2$. 
This, together with~\eqref{eq:btVSrt}, proves the
claim~\eqref{eq:twoconditions}.
The bound~\eqref{eq:nBound}, together with~\eqref{eq:boundnorms}
and~\eqref{eq:twoconditions}, 
thus implies that
\begin{equation}\label{eq:nDX}
	n \geq DX/4
\end{equation}
with
\[
  X := \log_2\left(\max\{\lambda + \lambda^{-1},N,
  \nTrOneStar{y,A},\condfsul(y,A)\}\right).
\]
Because $\log _2\max\{\alpha,\beta\}=\Omega(\log_2\alpha+\log_2\beta)$
we have 
\begin{equation}\label{eq:X2}
  X=\Omega(\log_2((\lambda + \lambda^{-1})N\nTrOneStar{y,A}\condfsul(y,A))
\end{equation}
and
\begin{align}\label{eq:X1}
  X=\Omega(\log_2((\lambda + \lambda^{-1})N\nTrOneStar{y,A}) 
  &=\Omega(\log_2((\lambda + \lambda^{-1})N\nTrOneStar{y,A}^2))\notag \\
  &= \Omega\left(\log_2\frac{C}{\delta^{1/4}}\right).
\end{align}
By Proposition~\ref{proposition:rhofssigmaub}(1), for
$\sigma_1(y,A)  < \lambda/4$, we have 
$\log_2(\sigma_1(y,A))\ge \log_2(\bt)
- \log_2\left(\frac{4\nmax{A}}{\lambda}\right)$. Hence, 
using that $\nTrOneStar{y,A}\ge \|A\|_{\max}$ and 
that $\condfsul(y,A)=\bt^{-1}$,
\begin{eqnarray*}
 \log_2(\sigma_1(y,A)) + n &\underset{\eqref{eq:nDX}}{\geq}& 
 -\log_2\condfsul(y,A) - \log_2\left(\frac{4\nmax{A}}{\lambda}\right)
 + DX/4\\
 &\underset{\eqref{eq:X2}}{\geq}& \Oh(X)-DX/4 \ge KX 
 \underset{\eqref{eq:X1}}{\geq} 2\log_2(C) + n
\end{eqnarray*}
where the constant $K$ is chosen so that $KX \geq 2 \log_2(C) + 2n$
(in particular, this implies the last inequality) and then the
penultimate inequality is ensured by choosing $D$ sufficiently
large. If instead $\sigma_1(y,A) \geq \lambda/4$, we use that
$X\ge |\log_2 \lambda|$ to deduce that
\begin{equation*}
  \log_2(\sigma_1(y,A)) + n \geq \log_2(\lambda/4) + DX/4 \ge KX
  \geq 2\log_2(C) + n
\end{equation*}
the last part of the reasoning being identical to the above.
In both cases, we get $\sigma_1(y,A)\ge C^2$. 

We next use Proposition~\ref{proposition:rhofssigmaub}(2--3) to get,
reasoning as above, 
\begin{align*}
  \log_2(\sigma_2(y,A)) + n &\geq -2\log_2\condfsul(y,A) + \frac{DX}{4}
  \geq \frac{KX}{2}\ge \log_2(C/\delta^{\frac{1}{4}}) = \log_2(C) + n\\
  \log_2(\sigma_3(y,A)) + n &\geq -\log_2\condfsul(y,A)
  - \log_2(\nmax{A}) + \frac{DX}{4}\ge KX \geq 2\log_2(C) + n
\end{align*}
which shows $\sigma_2(y,A)\ge C$ and $\sigma_3(y,A)\ge C^2$, as we wanted
to prove. From the definition of $\sigma$ we
conclude that $\sigma(y,A) \geq C^2$. 

All that remains is to show that
$\nTrMax{y,A} (mN)^{-1/2} \geq 2\delta$. But for any $D\geq 2$, it follows
from~\eqref{eq:nBound} that
\[
  n-1 \geq n/2\geq \log_2 N \geq \log_2((mN)^{1/2})
  \ge \log_2((mN)^{1/2}) -\log_2\nTrMax{y,A} 
\]
the third inequality as $m\le N$ and the last as $\nTrMax{y,A}\ge 1$. 
This completes the proof by noting that $2^{1-n} = 2\delta.$
\end{proof}

\subsubsection{A bound on the number of arithmetic operations performed
at the $n$th iteration of the repeat loop}

We analyse each section of the {\bf repeat} loop line by line, calculating
the number of operations performed in big $\Oh$ notation. The first two
lines, $n:=n+1$ and $\delta := \delta/16$, can be done using two
arithmetic operations.

Next we analyse the calls to oracles and subroutines. Firstly, 
$\orFindMatrix^U(n)$ and $\orFindVector^b(n)$ are executed with cost
$\Oh(mN(\log_2(\nmax{U}+1)+n))$ and
$\Oh(m(\log_2(\|b\|_{\infty}+1)+n))$, respectively. Both quantities are
bounded by $\Oh(N^2(\log_2\nTrMax{b,U}+n))$.
The call to $\orULasso(y,A,\lambda)$ takes $\Oh(N^{3}(p+\log_2(N))$ 
operations where $p$ is the largest number of bits for the entries of 
$y,A$ and $\lambda$. Note that
$p\leq \max\{\log_2(\nmax{A}+1)+n,
\log_2(\|y\|_{\infty}+1)+ n,p(\lambda)\} \leq
\max\{\log_2(3\nTrMax{b,U})+n,p(\lambda)\}$
and thus the runtime of 
$\orULasso(y,A,\lambda)$ is
$\Oh[N^{3}(\max\{\log_2(\nTrMax{b,U})+n,p(\lambda)\}+ \log_2(N))]$.
Calculating $\delta^{1/4}$ can be done in one arithmetic operation.
We can calculate $H$ and $G$ in $\Oh(mN) = \Oh(N^2)$ operations. 
To calculate $C$ requires five multiplications and one addition of already
calculated quantities and thus takes $\Oh(1)$ time. Finally, the call to 
$\orSigma$ takes $\Oh(N^3)$ operations (Proposition~\ref{prop:Sigma}).

Thus the total runtime of the $n$th iteration of the loop is 
\begin{align*}
 &\Oh(N^2(\log_2\nTrMax{b,U} +n)\\
 &+ \Oh[N^{3}(\max\{\log_2(\nTrMax{b,U})+n+1,p(\lambda)\}
 + \log_2(N))] + \Oh(N^3)
\end{align*}
which is equal to 
\[ 
 \Oh[N^{3}(\max\{\log_2(\nTrMax{b,U})+n,p(\lambda)\} + \log_2(N))].
\]

\subsubsection{The overall runtime of the algorithm}

By Lemma~\ref{lemma:boundOnIterations}, 
the number of iterations of $\orFSUL$ is bounded above by some $r$ with 
\[
 r=\Oh(\left\lceil \log_2\left(\max\{\lambda + \lambda^{-1},N,
 \nTrMax{b,U},\condfsul(b,U)\}\right)\right\rceil).
\] 
Thus, the total runtime is bounded above by 
\begin{align*}
 &\sum_{n=1}^{r}
 \Oh[N^{3}(\max\{\log_2(\nTrMax{b,U})+n+1,p(\lambda)\} + \log_2(N))]\\
 &=\Oh\left\{N^{3}\left[r\log_2(\nTrMax{b,U})+r^2+rp(\lambda)
 + r\log_2(N)\right]\right\}.
\end{align*}
Clearly, $\log_2(N)=\Oh(r)$ and $r\log_2(\nTrMax{b,U})=\Oh(r^2)$.
Hence the total runtime is bounded above by $\Oh(N^{3}r^2 + N^3rp(\lambda))$
i.e.
\newcommand{\finalComplexityBound}{
\Oh\bigl\{N^{3}&\left[\log_2\left(N^2(\lambda + \lambda^{-1})^2 \nTrMax{b,U}^2
 \condfsul(b,U)\right)\right]^2\\
 +& N^3\log_2\left(N^2(\lambda + \lambda^{-1})^2 \nTrMax{b,U}^2
 \condfsul(b,U)\right)p(\lambda)\bigr\}
}
\begin{align*}
 \finalComplexityBound
\end{align*}

Coupled with the earlier argument to show correctness for the algorithm,
we have completed the proof of Theorem~\ref{thm:fsulcomp}~parts~(1) and~(2).
\eproof

\subsection{Proof of Theorem~\ref{thm:fsulcomp}, part~(3)}

We start by proving the following: 
\begin{lemma}\label{lemma:PeturbToGetEitherMin}
Let $(y,U) \in \Omega$ and $x \in \Sol^{\mathsf{ms}}(y,U)$. Then if $E$ is the diagonal matrix with entries
$(\indic_{\{1 \notin \supp(x)\}},
\indic_{\{2 \notin \supp(x)\}} ,
\dotsc,\indic_{\{N \notin \supp(x)\} })$ where $\indic_{\{ i \notin \supp(x) \} }$ is $1$ if $i \notin \supp(x)$ and $0$ otherwise and if $\delta\in(0,1)$ we have
then $\mUL(y,U(\Id-\delta E^1))=\{x\}$.
\end{lemma}

\begin{proof}[Proof of Lemma~\ref{lemma:PeturbToGetEitherMin}]
This proof is similar to~\cite[Lemma 17.3]{opt_big}. We show that $x$ is the unique vector in $\mUL(y,U-\delta U E)$.

Suppose that $v$ is such that $\|(U-\delta U E)v - y\|^2_2
+ \lambda \|v\|_1 \leq \|(U-\delta U E) x - y\|_2^2 + \lambda \|x\|_1$.
We claim that $\supp(v) \subseteq \supp(x)$. Otherwise, if we
let $\hat v$ be defined by $\hat v_i = v_i$ whenever $i \in \supp(x)$
and $\hat v_i = (1-\delta) v_i$ whenever $i \notin \supp(x)$,
we obtain $(U -UE^1)v = U\hat v$ and $\|\hat v\|_1 < \|v\|_1$
(the strict inequality follows from $\supp(v) \not\subseteq \supp(x)$)
and hence (setting $f(w):= \|Uw-y\|^2_2 + \lambda \|w\|_1$)
\[
f(\hat v) < \|(U - \delta UE)v\|^2_2 + \lambda \|v\|_1
\leq \|(U - \delta U E) x- y\|_2^2 + \lambda \|x\|_1 = f(x)
\]
contradicting $x \in \mUL(y,U)$. Hence $\supp(v) \subseteq \supp(x^1)$.
But then $Uv = (U - \delta UE)v$ and so 
\[
f(v) = \|(U - \delta UE)v\|^2_2 + \lambda \|v\|_1
\leq \|(U - \delta UE) x - y\|_2^2 + \lambda \|x^1\|_1 = f(x^1)
\]
so $v \in \mUL(y,U)$ with $\supp(v) = \supp(x)$. It follows from
$x \in \Sol^{\mathsf{ms}}(y,U)$ that $v = x$ and so $x$ is the
unique vector in $\mUL (y,U-UE^1)$.
\end{proof}

\begin{proof}[Proof of Theorem~\ref{thm:fsulcomp}, part~(3)]
We will consider two cases: firstly, the case where there exists
$(b,U) \in \Omega$ with $|\Xi(b,U)| > 1$ and secondly the case where
all $(b,U) \in \Omega$ have $|\Xi(b,U)| = 1$.

\textbf{Case 1: There exists $(b,U) \in \Omega$ with $|\Xi(b,U)| > 1$.}

By Lemma~\ref{lemma:multipleSolutionsImpliesMS}, there exists
$x^1,x^2 \in \mULMS(b,U)$ with $\supp(x^1) \neq \supp(x^2)$. We now
use Lemma~\ref{lemma:PeturbToGetEitherMin} to see that for each
$n \in \mathbb{N}$ and for $i=1,2$, there exists
$b^{i,n} \in \real^m,U^{i,n} \in \real^{m \times N}$ so that
$\Xi(b^{i,n},U^{i,n}) = \supp(x^i)$ and $\|b^{i,n} -
b\|_{\infty} \leq 2^{-n}$, $\nmax{U^{i,n} - U} \leq 4^{-n}$. Since
$\Omega$ is open, we can pass to a subsequence and assume that
$(b^{i,n},U^{i,n}) \in \Omega$ for all $n \in \mathbb{N}$ and
$i \in \{1,2\}$. The conclusion now follows by an application of
Proposition~\ref{prop:DrivingNegativeProposition}, where we choose
$\iota^1_n = (b^{1,n},U^{1,n})$, $\iota^2_n = (b^{2,n},U^{2,n})$ and
$\iota^0 = (b,U)$.

\textbf{Case 2: All $(b,U) \in \Omega$ have $|\Xi(b,U)| = 1$.}

Let $\Xi(y,A) = s^0$. By assumption, $\condfsul(y,A) = \infty$. Thus
there exists a sequence
$(y^n,A^n) \in \real^{m} \times \real^{m \times N}$ with $\|y^n -
y\|_{\infty} \leq 4^{-n}$, $\nmax{A^n - A} \leq 4^{-n}$ and
$\Xi(y^n,A^n) = s^{1,n}$ with $s^{1,n} \neq s^0$. Since $\Omega$ is
open we can assume by potentially passing to a subsequence that
$(y^n,A^n) \in \Omega$. Furthermore, since the number of possible
supports for vectors in $\real^N$ is finite, $\{s^{1,n} \, \vert \,
n \in \mathbb{N}\}$ is finite and we can again pass to a subsequence
to assume that $s^{1,n} = s^1$. Once again, the conclusion now follows
by an application of
Proposition~\ref{prop:DrivingNegativeProposition}, where we choose
$\iota^1_n = (y^{1,n},A^{1,n})$ and $\iota^0 = \iota^2_n = (y,A)$.
\end{proof}

\section{Tools from the SCI hierarchy and GHA}\label{sec:SCI}
In this section we explain the SCI framework in the context of the
LASSO problem. The purpose of introducing this framework is to ensure
that  Theorem~\ref{thm:fsulcomp}, part~(3)] is proven with as much
generality as possible, taking into account the wide variety of
computational models (e.g. Turing, BSS, Von-Neumann, etc.) that are
prevalent in the optimisation literature. We start by defining the
LASSO feature selection as a `Computational problem' following the
setup of~\cite{opt_big}.

\newcommand{\nh}{\mathsf{NH}}
\begin{definition}[The LASSO computational problem]
\label{definition:LassoComputationalProblem}
For some set $\Omega \subset \real^{m} \times \real^{m \times N}$,
which we call the \emph{input} set, the \emph{LASSO
computational problem on} $\Omega$ is the collection
$\{\Xi,\Omega, \mathbb{B}^N,\Lambda\}$ where $\Xi:\Omega\to2^{\mathbb{B}^N}$
is defined
as in~\eqref{eq:LassoComp} and
\[
 \Lambda = \{f^{\mathrm{vec}}, f^{\mathrm{mat}}\} \text{ with }
 f^{\mathrm{vec}}:\Omega\to\real^m, f^{\mathrm{mat}}:\Omega\to\real^{m\times N}
\]
are defined by $f^{\mathrm{vec}}(y,A) = y$ and
$f^{\mathrm{mat}}(y,A) = A$ for all $(y,A) \in \Omega$.
\end{definition}
\newcommand{\dyadic}{\mathbf{D}}
\newcommand{\pr}{\mathbb{P}}

We want to generalise the LASSO computational problem so that we work
with \emph{inexact inputs}. To do so, 
we will consider the collection
of all functions $f_{n}^{\mathrm{vec}}: \Omega \to \real^m$
and $f_{n}^{\mathrm{mat}}: \Omega \to \real^{m\times N}$ satisfying
\begin{align}\label{eq:Lambda_y}
\|f_{n}^{\mathrm{vec}}(y,A) - y\|_{\infty} &\leq 2^{-n}, \quad  \nmax{f_{n}^{\mathrm{mat}}(y,A)-A} \leq 2^{-n} 
\end{align}
for all $(y,A) \in \Omega$.
To handle inexact input we follow \cite{opt_big} and replace the exact input set, $\Omega$, by the inexact input set
$\tilde{\Omega}$
to form the \emph{inexact LASSO computational problem}.

\begin{definition}[Inexact LASSO computational problem]
\label{definition:Omega_tilde_Delta_m}
The \emph{inexact LASSO computational problem on $\Omega$} (ILCP) is the
quadruple
$
\{\tilde \Xi,\tilde \Omega,\mathbb{B}^N,\tilde \Lambda\},
$
where
\begin{equation}
\begin{split}
\tilde \Omega = \big\{ &\tilde \iota =
\{(f_{n}^{\mathrm{vec}}(\iota),f_{n}^{\mathrm{mat}}(\iota)\}_{n\in\nat}
\mid \iota = (y,A) \in \Omega \text { and }\\
&f_{n}^{\mathrm{vec}}: \Omega \to \real^m,
f_{n}^{\mathrm{mat}}:\Omega\to\real^{m\times N}
\text{ satisfy~\eqref{eq:Lambda_y} respectively}\big\}
\end{split}
\end{equation}
It follows from~\eqref{eq:Lambda_y} that 
there is a unique
$\iota =(y,A)\in\Omega$ for which
\[
 \tilde \iota = \big\{(f_{n}^{\mathrm{vec}}(\iota),
 f_{n}^{\mathrm{mat}}(\iota))\big\}_{n\in\nat}.
\]
We say that this $\iota\in\Omega$ \emph{corresponds} to
$\tilde\iota\in\tilde\Omega$ and we set
$\tilde \Xi: \tilde \Omega \rightrightarrows \mathbb{B}^N$ so
that $\Xi(\tilde \iota) = \Xi(\iota)$, and
$\tilde \Lambda = \{\tilde f_{n}^{\mathrm{vec}},\tilde
f_{n}^{\mathrm{mat}}\}_{n \in \nat}$,
with $\tilde f_{n}^{\mathrm{vec}}(\tilde\iota) = f_{n}(\iota)$,
$\tilde f_{n}^{\mathrm{mat}}(\tilde\iota) = f_{n}^{\mathrm{mat}}(\iota)$
where $\iota$ corresponds to $\tilde \iota$.
\end{definition}

Part~(3) in Theorem~\ref{thm:fsulcomp} is a negative result (it states
that a computational problem cannot be solved). To make it as general
as possible we want it to apply to a broad class of algorithms. These
are defined, roughly speaking, by a single requirement namely, that
for a given $\tilde \iota \in \tilde\Omega$, the algorithm finds any
member of $\tilde \Xi(\tilde \iota)$ by accessing finitely many values
of $\tilde f(\tilde \iota)$ where $\tilde
f \in \tilde \Lambda$. This captures the idea that
any algorithm that solves the ILCP should return an
answer by accessing arbitrarily many (but finitely many)
approximations to the true input. Note, the algorithm should work for any
choice of approximations and any input but it may 
return different results depending on which approximations it sees.

The type of algorithm described above are called \emph{general algorithms}.
As noted in~\cite{opt_big}, the purpose of a general algorithm is to have
a definition that encompasses any model of computation. This ensures 
that Theorem~\ref{thm:fsulcomp}(3) holds in all relevant computational
models considered by the optimisation community (e.g. Turing, BSS).

\begin{definition}[General Algorithms for the ILCP]
\label{definition:Algorithm}
A \emph{general algorithm} for 
$\{\tilde \Xi,\tilde \Omega,\mathbb{B}^N,\tilde \Lambda\}$,
is a mapping 
$\Gamma:\tilde\Omega\to\mathbb{B}^N$
such that, for every $\tilde \iota\in\tilde \Omega$,
the following conditions hold:
\begin{enumerate}[label=(\roman*)]
\item there exists a nonempty subset of evaluations
 $\Lambda_\Gamma(\tilde \iota) \subset\tilde\Lambda $ with 
 $|\Lambda_\Gamma(\tilde \iota)|<\infty$\label{property:AlgorithmFiniteInput},
\item the action 
of $\,\Gamma$ on $\tilde\iota$ is uniquely determined by
 $\{f(\tilde \iota)\}_{f \in \Lambda_\Gamma(\tilde\iota)}$,
 \label{property:AlgorithmDependenceOnInput}
\item for every $\iota^{\prime} \in\Omega$ such that
 $f(\iota^\prime)=f(\tilde\iota)$
 for all $f\in\Lambda_\Gamma(\tilde\iota)$, it holds that
 $\Lambda_\Gamma(\iota^{\prime})=\Lambda_\Gamma(\tilde\iota)$.
 \label{property:AlgorithmSameInputSameInputTaken}
\end{enumerate}
\end{definition}

This will prove useful when we come to define
randomised general algorithms to prove the random setting statement 
in Theorem~\ref{thm:fsulcomp}(3). We can thus proceed to define
randomised general algorithms.

\begin{definition}[Randomised General Algorithm for the ILCP]
\label{definition:ProbablisticAlgorithm}
A \emph{randomised general algorithm} (RGA) for 
$\{\tilde \Xi,\tilde \Omega,\mathbb{B}^N,\tilde \Lambda\}$
is a collection $X$ of general algorithms
$\Gamma:\tilde\Omega\to\mathbb{B}^N$, a sigma-algebra
$\mathcal{F}$ on $X$, and a family of probability measures
$\{\mathbb{P}_{\iota}\}_{\iota \in \Omega}$ on
$\mathcal{F}$ such that the following conditions hold:
\begin{enumerate}[label=(P\roman*)]
\item For each $\iota \in \tilde\Omega$, the mapping
$\gprob_{\iota}:(X,\mathcal{F}) \to (\mathbb{B}^N, \mathcal{B})$
defined by $\gprob_{\iota}(\Gamma) = \Gamma(\iota)$ is a random variable, where
$\mathcal{B}$ is the Borel sigma-algebra on $\mathbb{B}^N$.
\label{property:PAlgorithmMeasurable}
\item For each $n \in \mathbb{N}$ and $\iota \in \tilde\Omega$, we have
$\lbrace \Gamma \in X \, \vert \, T_{\Gamma}(\iota) \leq n \rbrace
\in \mathcal{F}$,
where 
\[
 T_{\Gamma}(\tilde \iota) :=\sup\lbrace m \in \mathbb{N} \, \vert \,
 \text{ either } f_{m}^{\mathrm{vec}}  \in \Lambda_{\Gamma}(\tilde\iota)
 \text{ or } f_{m}^{\mathrm{mat}} \in \Lambda_{\Gamma}(\tilde\iota) \rbrace.
\]

\label{property:PAlgorithmRTMeasurable}

\item For all $\iota_1,\iota_2 \in \tilde\Omega$ and $E \in \mathcal{F}$
so that, for every $\Gamma \in E$ and every $f \in \Lambda_{\Gamma}(\iota_1)$,
we have $f(\iota_1) = f(\iota_2)$, it holds that
$\mathbb{P}_{\iota_1}(E) = \mathbb{P}_{\iota_2}(E)$.
\label{property:PAlgorithmConsistent}
\end{enumerate}
\end{definition}

\begin{remark}
The quantity $T_{\Gamma}(\tilde \iota)$ is known as the minimum amount of
input information in~\cite{opt_big}.
\end{remark}

\begin{definition}[Halting randomised general algorithms]
\label{def:halting_randomised}
A randomised general algorithm $\gprobh$ for a computational problem
$\{\Xi,\Omega,\mathcal{M},\Lambda\}$ is called a
\emph{halting randomised general algorithm} (hRGA) if
$\mathbb{P}_{\iota}(\gprobh_{\iota} = \nh) = 0$, for all $\iota\in\Omega$.
\end{definition}
\newcommand{\opBall}[2]{\mathcal{B}_{#1}{\left(#2\right) }}
\newcommand{\clBall}[2]{\overline{\mathcal{B}}_{#1}{\left(#2\right) }}
\newcommand{\strbdepsph}{\epsilon_{\mathbb{P}h\mathrm{B}}^{\mathrm{s}}(\mathrm{p})}
\newcommand{\strbdepsp}{\epsilon_{\mathbb{P}\mathrm{B}}^{\mathrm{s}}(\mathrm{p})}
\newcommand{\disM}{\mathrm{dist}_{\mathcal{M}}}
\newcommand{\bin}{\mathbb{B}}
We make use of the following propositions, taken from~\cite{opt_big}
and simplified for
the specific problem under consideration in this paper.

\begin{proposition}\label{prop:DrivingNegativeProposition}
(Simplified from~\cite{opt_big}, Proposition 9.5)
Suppose that a subset $\Omega$ of
$\cup_{N=1}^{\infty}\cup_{m=1}^{N} \real^m \times \real^{m \times N}$ is such
that there exists two sequences
$\{\iota^1_n\}_{n=1}^{\infty}, \{\iota^2_n\}_{n=1}^{\infty}\subset \Omega$
and an $\iota^0 \in  \Omega$ satisfying the following conditions:
\begin{enumerate}[leftmargin=8mm, label=(\alph*)]
\item There are disjoint sets $S^1,S^2 \subset \mathbb{B}^N$ with
$\Xi(\iota^i_n)\subset S^i$
for $i=1,2$ where $\Xi$ is defined as in
Definition~\ref{definition:LassoComputationalProblem}.
\label{property:MinimisersOfNonZero}
\item $\iota^i_n = (y^{i,n},A^{i,n})$ and $\iota^0 = (y,A)$ with
$\|y^{i,n} - y\|_{\infty}\leq 1/4^n$ and $\nmax{A^{i,n} - A} \leq 1/4^n$
for all $n \in\mathbb{N}$ and $i=1,2$. \label{property:Del1Info}
\end{enumerate}
Then each of the following holds: 

\begin{enumerate}
\item For any general algorithm $\Gamma: \tilde \Omega \to \bin^N$,
there must exist a
$\tilde \iota  \in \tilde\Omega$ so that
$\Gamma(\tilde{\iota}) \neq \tilde \Xi(\tilde \iota)$.
\item For any $p>0$ and halting randomised general algorithm $\gprobh$,
there must exist $\tilde \iota$ so that the probability that
$\gprobh(\tilde\iota) \neq \tilde \Xi(\tilde \iota)$ is at least $1/2 - p$.
\end{enumerate}
\end{proposition}		
\bibliographystyle{abbrv}
\bibliography{condReferencesDet}

\begin{thebibliography}{10}

\bibitem{CSBook}
B.~Adcock and A.~C. Hansen.
\newblock {\em Compressive Imaging: Structure, Sampling, Learning}.
\newblock Cambridge University Press, 2021.

\bibitem{Lotz2014}
D.~Amelunzen, M.~Lotz, M.~B. McCoy, and J.~A. Tropp.
\newblock Living on the edge : phase transitions in convex programs with random
  data.
\newblock {\em Information and Inference}, 3(3):224--294, June 2014.

\bibitem{paradox22}
V.~Antun, M.~J. Colbrook, and A.~C. Hansen.
\newblock Proving existence is not enough: Mathematical paradoxes unravel the
  limits of neural networks in artificial intelligence.
\newblock {\em SIAM News}, 55(04):1--4, May 2022.

\bibitem{antun2020instabilities}
V.~Antun, F.~Renna, C.~Poon, B.~Adcock, and A.~C. Hansen.
\newblock On instabilities of deep learning in image reconstruction and the
  potential costs of {AI}.
\newblock {\em Proc. Natl. Acad. Sci. USA}, 117(48):30088--30095, 2020.

\bibitem{Arora2007}
S.~Arora and B.~Barak.
\newblock {\em Computational Complexity - A Modern Approach}.
\newblock Princeton University Press, 2009.

\bibitem{BCH2}
A.~Bastounis, F.~Cucker, and A.~C. Hansen.
\newblock When can you trust feature selection? -- ii: {O}n the effects of
  random data on condition in statistics and optimisation.
\newblock {\em Preprint}, 2023.

\bibitem{opt_big}
A.~Bastounis, A.~C. Hansen, and V.~{Vla\v{c}i\'{c}}.
\newblock The extended {S}male's 9th problem -- {O}n computational barriers and
  paradoxes in estimation, regularisation, computer-assisted proofs and
  learning.
\newblock {\em arXiv:2110.15734}, 2021.

\bibitem{Fista}
A.~Beck and M.~Teboulle.
\newblock A fast iterative shrinkage-thresholding algorithm for linear inverse
  problems.
\newblock {\em SIAM Journal on Imaging Sciences}, 2(1):183--202, 2009.

\bibitem{SCI}
J.~Ben-Artzi, M.~J. Colbrook, A.~C. Hansen, O.~Nevanlinna, and M.~Seidel.
\newblock Computing spectra -- {O}n the solvability complexity index hierarchy
  and towers of algorithms.
\newblock {\em arXiv:1508.03280v5}, 2020.

\bibitem{Ben_Artzi2022}
J.~Ben-Artzi, M.~Marletta, and F.~R\"osler.
\newblock Computing scattering resonances.
\newblock {\em Journal of the European Mathematical Society}, 2023.

\bibitem{Nemirovski_robust}
A.~Ben-Tal, L.~El~Ghaoui, and A.~Nemirovski.
\newblock {\em Robust Optimization}.
\newblock Princeton Series in Applied Mathematics. Princeton University Press,
  October 2009.

\bibitem{Nemirovski_robust2}
A.~Ben-Tal and A.~Nemirovski.
\newblock Robust solutions of linear programming problems contaminated with
  uncertain data.
\newblock {\em Mathematical Programming}, 88(3):411--424, 2000.

\bibitem{BCSS98}
L.~Blum, F.~Cucker, M.~Shub, and S.~Smale.
\newblock {\em Complexity and Real Computation}.
\newblock Springer-Verlag, 1998.

\bibitem{BSS89}
L.~Blum, M.~Shub, and S.~Smale.
\newblock On a theory of computation and complexity over the real numbers:
  {NP}-completeness, recursive functions and universal machines.
\newblock {\em BAMS}, 21:1--46, 1989.

\bibitem{Boyd}
S.~Boyd and L.~Vandenberghe.
\newblock {\em Convex Optimization}.
\newblock Cambridge University Press, New York, NY, USA, 2004.

\bibitem{Condition}
P.~B\"urgisser and F.~Cucker.
\newblock {\em Condition: The Geometry of Numerical Algorithms}.
\newblock Number 349 in Grundlehren der matematischen Wissenschaften. Springer
  Verlag, 2013.

\bibitem{Chambolle_2011}
A.~Chambolle and T.~Pock.
\newblock A first-order primal-dual algorithm for convex problems with
  applications to imaging.
\newblock {\em J. Math. Imaging Vis.}, 40(1):120--145, May 2011.

\bibitem{chambolle_pock_2016}
A.~Chambolle and T.~Pock.
\newblock An introduction to continuous optimization for imaging.
\newblock {\em Acta Numerica}, 25:161--319, 2016.

\bibitem{felipe_cond_01}
D.~Cheung and F.~Cucker.
\newblock A new condition number for linear programming.
\newblock {\em Mathematical Programming}, 91(1):163--174, 2001.

\bibitem{ChC02}
D.~Cheung and F.~Cucker.
\newblock Solving linear programs with finite precision: {I}. {C}ondition
  numbers and random programs.
\newblock {\em Math. Programming}, 99:175--196, 2004.

\bibitem{ChC05}
D.~Cheung and F.~Cucker.
\newblock A note on level-2 condition numbers.
\newblock {\em Journal of Complexity}, 21(3):314--319, 2005.

\bibitem{ChCP07}
D.~Cheung, F.~Cucker, and J.~Pe\~na.
\newblock On strata of degenerate polyhedral cones. {I}: Condition and distance
  to stratae.
\newblock {\em European Journal of Operational Research}, 198:23--28, 2009.

\bibitem{Choi}
C.~Choi.
\newblock 7 revealing ways {AI}s fail.
\newblock {\em IEEE Spectrum}, September, 2021.

\bibitem{Choi2}
C.~Choi.
\newblock Some {AI} systems may be impossible to compute.
\newblock {\em IEEE Spectrum}, March, 2022.

\bibitem{Matt1}
M.~Colbrook and A.~C. Hansen.
\newblock The foundations of spectral computations via the solvability
  complexity index hierarchy.
\newblock {\em Journal of the European Mathematical Society}, 2022 (online).

\bibitem{comp_stable_NN22}
M.~J. Colbrook, V.~Antun, and A.~C. Hansen.
\newblock The difficulty of computing stable and accurate neural networks: On
  the barriers of deep learning and smale's 18th problem.
\newblock {\em Proc.\ Natl.\ Acad.\ Sci.\ USA}, 119(12):e2107151119, 2022.

\bibitem{Cucker_Smale97}
F.~Cucker and S.~Smale.
\newblock Complexity estimates depending on condition and round-off error.
\newblock {\em Journal of the ACM}, 46(1):113--184, 1999.

\bibitem{AIM}
C.~Fefferman, A.~C. Hansen, and S.~Jitomirskaya, editors.
\newblock {\em Computational mathematics in computer assisted proofs}, American
  Institute of Mathematics Workshops. American Institute of Mathematics, 2022.
\newblock Available online at \\
  \url{https://aimath.org/pastworkshops/compproofsvrep.pdf}.

\bibitem{gazdag2022generalised}
L.~E. Gazdag and A.~C. Hansen.
\newblock Generalised hardness of approximation and the {SCI} hierarchy -- {On}
  determining the boundaries of training algorithms in {AI}.
\newblock {\em arXiv:2209.06715}, 2022.

\bibitem{vNGo51}
H.~Goldstine and J.~von Neumann.
\newblock Numerical inverting matrices of high order, {II}.
\newblock {\em Proc. Amer. Math. Soc.}, 2:188--202, 1951.

\bibitem{Anders2}
N.~M. Gottschling, V.~Antun, A.~C. Hansen, and B.~Adcock.
\newblock The troublesome kernel -- on hallucinations, no free lunches and the
  accuracy-stability trade-off in inverse problems.
\newblock 2023.

\bibitem{Hales1}
T.~C. Hales.
\newblock A proof of the {K}epler conjecture.
\newblock {\em Ann. of Math. (2)}, 162(3):1065--1185, 2005.

\bibitem{Hales2}
T.~C. Hales and et~al.
\newblock A formal proof of the kepler conjecture.
\newblock {\em Forum of Mathematics, Pi}, 5:e2, 2017.

\bibitem{hammernik2018learning}
K.~Hammernik, T.~Klatzer, E.~Kobler, M.~P. Recht, D.~K. Sodickson, T.~Pock, and
  F.~Knoll.
\newblock Learning a variational network for reconstruction of accelerated
  {MRI} data.
\newblock {\em Magnetic Resonance in Medicine}, 79(6):3055--3071, 2018.

\bibitem{Hansen_JAMS}
A.~C. Hansen.
\newblock On the solvability complexity index, the {$n$}-pseudospectrum and
  approximations of spectra of operators.
\newblock {\em Journal of the American Mathematical Society}, 24(1):81--124,
  2011.

\bibitem{Tibshirani_Book}
T.~Hastie, R.~Tibshirani, and J.~Friedman.
\newblock {\em The Elements of Statistical Learning}.
\newblock Springer Series in Statistics. Springer New York Inc., New York, NY,
  USA, 2001.

\bibitem{SLSBook}
T.~Hastie, R.~Tibshirani, and M.~Wainwright.
\newblock {\em {Statistical Learning with Sparsity: The Lasso and
  Generalizations (Chapman \& Hall/CRC Monographs on Statistics \& Applied
  Probability)}}.
\newblock Chapman and Hall/CRC, May 2015.

\bibitem{heaven2019deep}
D.~Heaven et~al.
\newblock Why deep-learning {AIs} are so easy to fool.
\newblock {\em Nature}, 574(7777):163--166, 2019.

\bibitem{Higham96}
N.~J. Higham.
\newblock {\em Accuracy and Stability of Numerical Algorithms}.
\newblock Society for Industrial and Applied Mathematics, Philadelphia, PA,
  USA, 2nd edition, 2002.

\bibitem{jin17}
K.~H. Jin, M.~T. McCann, E.~Froustey, and M.~Unser.
\newblock Deep convolutional neural network for inverse problems in imaging.
\newblock {\em IEEE Transactions on Image Processing}, 26(9):4509--4522, 2017.

\bibitem{Juditsky_2012}
A.~Juditsky, F.~Kilin{\c{c}}{-}Karzan, A.~Nemirovski, and B.~Polyak.
\newblock {Accuracy guaranties for $\ell_{1}$ recovery of block-sparse
  signals}.
\newblock {\em The Annals of Statistics}, 40(6):3077 -- 3107, 2012.

\bibitem{Knuth76}
D.~Knuth.
\newblock Big omicrom and big omega and big theta.
\newblock {\em SIGACT News}, pages 18--24, Apr.-Jun. 1976.

\bibitem{Lotz2020}
M.~Lotz, D.~Amelunxen, and J.~Walvin.
\newblock Effective condition number bounds for convex regularization.
\newblock {\em IEEE Transactions on Information Theory}, Jan. 2020.

\bibitem{mccann2017convolutional}
M.~T. McCann, K.~H. Jin, and M.~Unser.
\newblock Convolutional neural networks for inverse problems in imaging: {A}
  review.
\newblock {\em IEEE Signal Process Magazine}, 34(6):85--95, 2017.

\bibitem{MonteiroON3L}
R.~D.~C. Monteiro and I.~Adler.
\newblock Interior path following primal-dual algorithms. part ii: Convex
  quadratic programming.
\newblock {\em Mathematical Programming}, 44(1):43--66, May 1989.

\bibitem{DezFaFr-17}
S.~Moosavi-Dezfooli, A.~Fawzi, O.~Fawzi, and P.~Frossard.
\newblock Universal adversarial perturbations.
\newblock In {\em IEEE Conference on computer vision and pattern recognition},
  pages 86--94, July 2017.

\bibitem{TopologicalVectorSpaces}
L.~Narici and E.~Beckenstein.
\newblock {\em Topological Vector Spaces}.
\newblock Chapman and Hall/{CRC}, July 2010.

\bibitem{NemirovskiLRob}
A.~Nemirovski.
\newblock { Lectures on Robust Convex Optimization}.
\newblock Available online at \url{https://www2.isye.gatech.edu/~nemirovs/},
  2009.

\bibitem{Nesterov_Nemirovski_Acta}
Y.~E. Nesterov and A.~Nemirovski.
\newblock On first-order algorithms for l1/nuclear norm minimization.
\newblock {\em Acta Numer.}, 22:509--575, 2013.

\bibitem{Pena2001}
J.~Pe{\~n}a.
\newblock Conditioning of convex programs from a primal-dual perspective.
\newblock {\em Mathematics of Operations Research}, 26(2):206--220, 2001.

\bibitem{Pena2002}
J.~Pe{\~n}a.
\newblock Two properties of condition numbers for convex programs via
  implicitly defined barrier functions.
\newblock {\em Mathematical Programming}, 93(1):55--75, 2002.

\bibitem{EU_Commission_2020}
H.~R, J.~H, and S.~M. JI.
\newblock Robustness and explainability of artificial intelligence.
\newblock (KJ-NA-30040-EN-N (online)), 2020.

\bibitem{renegar1988polynomial}
J.~Renegar.
\newblock A polynomial-time algorithm, based on newton's method, for linear
  programming.
\newblock {\em Mathematical Programming}, 40(1-3):59--93, 1988.

\bibitem{Renegar94}
J.~Renegar.
\newblock Is it possible to know a problem instance is ill-posed?
\newblock {\em J.of Complexity}, 10:1--56, 1994.

\bibitem{Renegar2}
J.~Renegar.
\newblock Incorporating condition measures into the complexity theory of linear
  programming.
\newblock {\em SIAM Journal on Optimization}, 5(3):506--524, 1995.

\bibitem{Renegar1}
J.~Renegar.
\newblock Linear programming, complexity theory and elementary functional
  analysis.
\newblock {\em Mathematical Programming}, 70(1):279--351, 1995.

\bibitem{Renegar96}
J.~Renegar.
\newblock Condition numbers, the barrier method, and the conjugate-gradient
  method.
\newblock {\em SIAM Journal on Optimization}, 6:879--912, 1996.

\bibitem{renegar2001mathematical}
J.~Renegar.
\newblock {\em A mathematical view of interior-point methods in convex
  optimization}, volume~3.
\newblock Siam, 2001.

\bibitem{MathFrontiersPerspectives}
S.~Smale.
\newblock Mathematical problems for the next century.
\newblock In V.~Arnold, M.~Atiyah, P.~Lax, and B.~Mazur, editors, {\em
  Mathematics: Frontiers and Perspectives}. American Mathematical Society,
  2000.

\bibitem{LassoStart}
R.~Tibshirani.
\newblock Regression shrinkage and selection via the lasso.
\newblock {\em Journal of the Royal Statistical Society, Series B},
  58:267--288, 1994.

\bibitem{LASSOUNIQUE}
R.~J. Tibshirani.
\newblock The lasso problem and uniqueness.
\newblock {\em Electron. J. Statist.}, 7:1456--1490, 2013.

\bibitem{Turing48}
A.~Turing.
\newblock Rounding-off errors in matrix processes.
\newblock {\em Quart. J. Mech. Appl. Math.}, 1:287--308, 1948.

\bibitem{vNGo47}
J.~von Neumann and H.~Goldstine.
\newblock Numerical inverting matrices of high order.
\newblock {\em Bull. Amer. Math. Soc.}, 53:1021--1099, 1947.

\bibitem{Wilkinson63}
J.~Wilkinson.
\newblock {\em Rounding Errors in Algebraic Processes}.
\newblock Prentice Hall, 1963.

\bibitem{wright2022optimization}
S.~Wright and B.~Recht.
\newblock {\em Optimization for Data Analysis}.
\newblock Cambridge University Press, 2022.

\bibitem{Mario_Lasso}
S.~J. Wright, R.~D. Nowak, and M.~A.~T. Figueiredo.
\newblock Sparse reconstruction by separable approximation.
\newblock {\em IEEE Transactions on Signal Processing}, 57(7):2479--2493, 2009.

\end{thebibliography}

\end{document}